\documentclass[final,3p,times]{elsarticle}
\usepackage{graphicx} % Required for inserting images

\usepackage{amsmath, amssymb, latexsym, amsthm, amscd, mathrsfs, stmaryrd}
\SetSymbolFont{stmry}{bold}{U}{stmry}{m}{n}
\usepackage{amsfonts, mathrsfs}
\usepackage{tikz}
\usetikzlibrary{arrows}
\usetikzlibrary{calc}
\usepackage{tikz-cd}
\usepackage{multirow}
\usepackage{yhmath}
\usepackage{mathtools}

\usepackage{tabularx}  %自己加的
\usepackage{array}
\usepackage{float} % 引入 float 包
\usepackage{bbding}
\usepackage[misc]{ifsym}
\newcommand{\A}{\mathcal{A}}

\newcommand{\f}{\frac{1}{2}} %1/2的快捷方式
   
\newcommand{\pa}{\partial_t}   

\newcommand{\3}[1]{\interleave #1 \interleave}%三竖线范数的简写设置
\newcommand{\n}[1]{\left\lVert #1 \right\rVert} %二竖线范数的简写设置
\newcommand{\inner}[1]{\left \langle #1 \right \rangle} %整点范数符号
\newcommand*{\dif}{\mathop{}\!\mathrm{d}}
\usepackage{amssymb}
\usepackage{amsmath}
\newtheoremstyle{noparens}%
  {}{}%
  {\itshape}{}%
  {\bfseries}{.}%
  { }%
  {\thmname{#1}\thmnumber{ #2}\mdseries\thmnote{ #3}}
\theoremstyle{noparens}%应用文献中的某个定理格式会好看
\newtheorem{example}{Example}%[section]
\newtheorem{remark}[subsubsection]{Remark}%[section]

\newtheorem{theorem}{Theorem}[section]
\newtheorem{lemma}{Lemma}[section]

\begin{document}

\begin{frontmatter}



\title{A Family of Block-Centered Schemes for Contaminant Transport Equations with Adsorption via Integral Method with Variational Limit}

%% Authors
% \author{} %% Author name

% \affiliation{organization={},%Department and Organization
            
\author[hrbeu]{He Liu}
\ead{liuhehe@hrbeu.edu.cn}

\author[hrbeu]{Xiongbo Zheng\corref{cor1}}
\ead{zhengxiongbo@hrbeu.edu.cn}

\author[hrbeu]{Xiaole Li}
\ead{lixiaole@hrbeu.edu.cn}

\author[hrbeu]{Mingze Ji}
\ead{jmzheu@163.com}

% 通讯作者说明
\cortext[cor1]{ Corresponding author.}
% 作者机构
\affiliation[hrbeu]{
  organization={School of Mathematical Sciences, Harbin Engineering University},
  addressline={Nantong Street 145},
  city={Harbin},
  postcode={150001},
  state={Heilongjiang},
  country={China}
}
%% Abstract
\begin{abstract}
This paper develops a class of high-order conservative schemes for contaminant transport with equilibrium adsorption, based on the Integral Method with Variational Limit on block-centered grids. By incorporating four parameters, the scheme can reproduce classical fourth-order compact schemes and further extend to sixth- and eighth-order accurate formulations, all within a unified framework. Under periodic boundary conditions, we analyze the stability, convergence, and mass conservation of the parameterized numerical scheme. Numerical experiments are then conducted to examine the impact of parameter variations on errors, explore the relationship between parameters and the fourth-, sixth-, and eighth-order schemes, and verify that the schemes’ high-order accuracy aligns with theoretical predictions. To enhance the applicability of the proposed method, we further develop two fourth-order compact boundary treatments that ensure uniform accuracy between boundary and interior regions. Numerical results confirm the effectiveness of the proposed schemes across various adsorption models
\end{abstract}

%% Keywords
\begin{keyword}
Contaminant transport\sep Integral method with variational limit \sep Block-centered grid \sep Convergence \sep Stability
\end{keyword}
\end{frontmatter}



% \section{Example Section}
% \label{sec1}
%% Labels are used to cross-reference an item using \ref command.
\section{Introduction}
\label{intro}
Contaminant transport with adsorption in porous media is a fundamental problem in environmental science and energy engineering, with wide applications in groundwater remediation~\cite{bear2013dynamics,du2010efficient,miller1998multiphase}, enhanced oil recovery~\cite{alt1985nonsteady,aziz1979petroleum}, and soil pollution control~\cite{amacher1988kinetics}. These processes feature a complex interplay between transport and adsorption, necessitating high-resolution numerical schemes for their reliable and accurate simulation.
\par
Among various numerical methods, the block-centered finite difference (BCFD) method~\cite{aziz1979petroleum,bear2012hydraulics,peaceman2000fundamentals} is widely adopted due to its capability to approximate both primary variables and fluxes simultaneously. It is regarded as a mixed finite element method \cite{raviart2006mixed}. Wheeler and Weiser \cite{weiser1988convergence} analyzed it and proved discrete \(L^2\) error estimates and second-order convergence on nonuniform grids for linear elliptic problems. Mehl et al. \cite{mehl2002development} applied BCFD with local mesh refinement to groundwater models. Li and Rui \cite{li2016two} developed a two-grid BCFD method for nonlinear non-Fickian transport, solving the nonlinear problem on a coarse grid and the linear one on a fine grid, achieving an error bound of $\mathcal{O}(t+h^2+H^3)$. Despite these advances, the classical BCFD schemes remain second-order accurate~\cite{arbogast1997mixed,liang2020conservative,rui2012block,rui2013block}, which limits their effectiveness in complex scenarios that require high-order accuracy.
\par
To improve accuracy, compact difference schemes~\cite{hirsh1975higher,lele1991compact} have gained attention for their high-order approximations with compact stencils. Some works combine compact schemes with BCFD to achieve both high accuracy and flux consistency. For example, Shi et al.~\cite{shi2022fourth} proposed a fourth-order compact BCFD method for nonlinear adsorption models, enabling simultaneous high-order approximation of both variables and fluxes. Ma et al.~\cite{ma2023fast} proposed a fourth-order compact BCFD scheme for variable-coefficient reaction-diffusion equations and verified its robustness under initial weak singularities.
These studies show that high-order methods combining compact schemes with block-centered conservative structures are theoretically reliable. 
However, most high-order BCFD schemes are designed for specific equations, resulting in a lack of structural uniformity and limited extensibility.
Therefore, how to construct a unified and tunable high-order scheme that works across multiple models is still under active investigation.
\par
To address the aforementioned challenges, this paper focuses on contaminant transport equations with adsorption and develops a family of unified high-order conservative schemes based on the Integral Method with Variational Limit (IMVL). 
The proposed approach delivers high-order accuracy for both variables and fluxes on block-centered grids.
Four adjustable parameters are introduced to flexibly generate numerical schemes of different orders while maintaining a unified discrete structure. Through parameter tuning, the method can reproduce classical fourth-order compact schemes and further construct a variety of high-order schemes of fourth, sixth, and eighth order. It provides various scheme options suitable for linear adsorption, Langmuir adsorption, and Freundlich adsorption models, effectively balancing model differences with structural consistency.
Furthermore, this paper investigates the impact of parameter choices on the accuracy of the schemes, revealing their relation to leading truncation errors, which provides guidance for parameter optimization and selection. To enhance boundary adaptability, Dirichlet boundary treatments are developed for two representative fourth-order compact schemes, ensuring uniform accuracy between interior and boundary points and preserving overall high-order convergence.
\par
In the theoretical analysis, this paper does not fix the parameters at specific values. Instead, it rigorously establishes convergence, stability, and mass conservation based on the general form of the scheme under periodic boundary conditions, ensuring that schemes with varying parameters retain consistent numerical properties.
In numerical validation, we first examine parameter effects on the error under periodic boundaries and find that the \(L^2\)-norm error is minimized when the scheme achieves eighth-order accuracy, consistent with the theory. Subsequently, numerical tests on fourth-, sixth-, and eighth-order schemes under various adsorption models both confirm the theoretical accuracy of the method and validate its mass conservation properties. Finally, numerical experiments of the two fourth-order schemes apply to three adsorption models under Dirichlet boundary conditions confirm that the schemes achieve fourth-order convergence, consistent with theoretical predictions.
In summary, this work proposes a unified framework with structural consistency, enabling the construction of high-order schemes across multiple models.
\par
The structure of the paper is as follows. Section 2 introduces the one-dimensional contaminant transport model with three adsorption cases and regularity assumptions. Section 3 presents the Integral Method with Variational Limit and constructs high-order schemes for both periodic and Dirichlet boundaries. Section 4 analyzes mass conservation, stability, and convergence under periodic conditions. Section 5 validates the method through numerical experiments, demonstrating its high accuracy.
\section{Mathematical Model}
\label{sec:1}
The contaminant transport in porous media is modeled by a mass-conservation equation that incorporates advection, molecular diffusion, mechanical dispersion, and chemical adsorption. The resulting one-dimensional governing equation is
\begin{equation} \label{eq-222}
c_t + \frac{\rho_b}{n}s_t + (uc - Dc_x)_x = f(x,t), \quad (x,t) \in (0, x_R) \times (0, T],
\end{equation}
where \( c(x,t) \) and \( s(x,t) \) denote the solute concentrations in the fluid and solid phases, respectively. The constants \( \rho_b \) and \( n \) represent the bulk density of the porous medium and its porosity. The fluid velocity is denoted by \( u(x) \), \( D(x) \) is the hydrodynamic dispersion coefficient, and \( f(x,t) \) is a source or sink term.
\par
The adsorption process can be described under either nonequilibrium or equilibrium assumptions. In nonequilibrium settings, the rate of adsorption is governed by first-order kinetics:
\begin{equation} \label{eq-ss}
\frac{\partial s}{\partial t} = k\left(\varphi(c) - s\right),
\end{equation}
where \( k \) is a reaction rate constant and \( \varphi(c) \) denotes the equilibrium sorbed concentration corresponding to the fluid concentration \( c \) according to the adsorption isotherm. Under equilibrium conditions, the solid-phase concentration \( s \) instantaneously attains equilibrium with the fluid-phase concentration \( c \), satisfying:
\[
    s = \varphi(c),
\] 
where \( \varphi(c) \) is specified by the chosen adsorption isotherm.

The most basic model is the linear isotherm, which assumes that the adsorbed concentration is directly proportional to the solute concentration\cite{seidel2004experimental}:
\[
\varphi(c) = K_d c,
\]
where \( K_d \) denotes the linear adsorption coefficient. This model is commonly applied to dilute systems where the solid-phase adsorption sites are far from saturation.

To capture nonlinear saturation effects, the Langmuir isotherm is often employed. It assumes monolayer adsorption on a homogeneous surface with a finite number of adsorption sites\cite{dawson1998analysis}:
\[
\varphi(c) = \frac{K_L S_m c}{1 + K_L c},
\]
where \( K_L \) is the Langmuir constant indicating the adsorption affinity, and \( S_m \) is the maximum adsorption capacity corresponding to full surface coverage.

For heterogeneous porous materials or non-ideal adsorption conditions, the Freundlich isotherm offers a more flexible empirical model\cite{solute}:
\[
\varphi(c) = K_F c^{\alpha}, \quad K_F > 0,\ \alpha > 0,
\]
where \( K_F \) is a proportionality constant and \( \alpha \) is a dimensionless exponent characterizing sorption intensity and nonlinearity. 
Notably, for the Freundlich isotherm with \(0 < \alpha < 1\), the derivative \(\varphi'(c)\) becomes unbounded as \(c \to 0\), leading to degeneracy.

In this work, we assume equilibrium adsorption governed by these isotherms. Furthermore, with constant \( \rho_b \) and \( n \), the time derivative of the adsorption term can be expressed as
\[
\frac{\rho_b}{n} s_t = \phi(c)_t, \quad \text{with} \quad \phi(c) := \frac{\rho_b}{n} \varphi(c).
\]
The resulting governing equation becomes
\begin{equation} \label{eq-3}
c_t + \phi(c)_t + (uc - Dc_x)_x = f(x,t), \quad (x,t) \in (0, x_R) \times (0, T].
\end{equation}
For the governing equation \eqref{eq-3}, we impose the following assumptions:
\begin{itemize}
\item[\textbf{(H1)}] The function \( \phi \) is Hölder continuous and monotone increasing on \(\mathbb{R}\).
\item[\textbf{(H2)}] The initial data \( c_0(x) \) is continuous on \( (0, x_R) \) and satisfies \( c_0(x) \ge 0 \).
\item[\textbf{(H3)}] The functions \( u(x) \in C^1(\mathbb{R}) \) and \( D(x) \) are both periodic in \( x \), and there exist constants \( D_*, D^* > 0 \) such that \(  D_* \le D(x) \le D^* \).
\item[\textbf{(H4)}] Define \( \psi(c) := c + \phi(c) \). Assume \( \psi \in C(\mathbb{R}) \), \( \psi(0) = 0 \), and that \( \psi \) is strictly increasing and uniformly monotone in the sense that
\begin{align} \label{eq-PP}
\frac{\psi(a)-\psi(b)}{a-b} \geq 1, \quad \forall a, b \in \mathbb{R},\ a \ne b.
\end{align}
\end{itemize}

\section{IMVL-Based Scheme for Contaminant Transport with Adsorption}
\subsection{The Integral Method with Variational Limit}
The Integral Method with Variational Limit (IMVL) \cite{luo2017fourth} is a high-order, accurate, and conservative discretization method. This section presents its theoretical foundation and spatial discretization.
\par
Consider a uniform partition of the domain \( [x_L, x_R] \) with grid size \( h = (x_R - x_L)/J \), where \( J \in \mathbb{N}_+ \), and define the grid points as \( x_i = x_L + i h \) for \( i = 1, \dots, J \). The spatial discretization is carried out as follows:
\begin{equation*}
\mathop{\int}_{xx}f(x):=\int_{x_i}^{x_i+\varepsilon_2 }\dif x_b\int_{x_i-\varepsilon_1}^{x_i}\dif x_a\int_{x_a}^{x_b}f(x)\dif x. 
\end{equation*}
where \( \varepsilon_1 \) and \( \varepsilon_2 \) are integration parameters. The function \( f(x) \) is expanded in a Taylor series about the \( x_i \):
\begin{align*}
    f(x)=\sum_{k}\frac{f^{(k)}}{k!}(x-x_i)^k.
\end{align*}
Substituting this expansion into the integral and evaluating term by term yields a sequence of coefficients. For convenience, we define
\begin{align*}
    e_k:= \mathop{\int}_{xx} \frac{(x-x_{i})^{k}}{k!} = \frac{\varepsilon_{1}\varepsilon_{2}^{k+2} + \varepsilon_{2}(-\varepsilon_{1})^{k+2}}{(k+2)!}, \quad
    B_{k}:= \mathop{\int}_{xx} \frac{\lvert{x-x_i}\rvert^{k}}{k!}  =\frac{\varepsilon_{1}\varepsilon_{2}^{k+2}+\varepsilon_{2}\varepsilon_{1}^{k+2}}{(k+2)!},
\end{align*}
where $ k\in \mathbb{N} $, and we have Lemma \ref{lem-2.1}.
%%%%%%%%%%%%%%%%  引理定义2. 1
\begin{lemma}[\cite{luo2017fourth}]\label{lem-2.1}
 Suppose $ f(x)\in{C^{K+1}}[x_L, x_R]$, we have 
\begin{align*}
    \mathop{\int}_{xx}f(x)=\sum_{k=0}^{K} {e}_kf_i^{(k)}+R,
\end{align*}   
where $\lvert{R}\rvert\leq\underset{x\in{[x_L, x_R]}}{\mathrm{max}}\lvert f^{(K+1)}(x)\rvert B_{K+1}$ and $ f_i^{(k)}=\frac{\dif^{k}f}{\dif x^{k}}\rvert_{x=x_i}, k\in\mathbb{N} $. 
\end{lemma}
To remove the common factor introduced by the integral and preserve the structure of the numerical scheme, the expression is divided by
\( e_0 := \int_{xx} 1 \). Accordingly, we define
\begin{align*} 
    \tilde{e}_k:= \frac{e_k}{e_0}= \frac{2(\varepsilon_2^{k+1}-(-\varepsilon_1)^{k+1})}{(\varepsilon _1+\varepsilon_2)(k+2)!}, \quad
    \tilde{B}_k:= \frac{B_k}{e_0}=\frac{2(\varepsilon_2^{k+1}+\varepsilon_1^{k+1})}{(\varepsilon_1+\varepsilon_2)(k+2)!}. 
\end{align*}
Therefore, based on Lemma \ref{lem-2.1}, we obtain
\begin{align*}
    \mathop{\int}_{xx}f(x){\div} \underset{xx}{\int}1 =\sum_{k=0}^{K}\tilde{e}_kf_i^{(k)}+{R},
\end{align*}  
where ${\lvert{R}}\rvert\leq\underset{x\in{[x_L, x_R]}}{\mathrm{max}}\lvert f^{(K+1)}(x)\rvert\tilde{ B}_{K+1}$.

To prepare for the scheme construction, the derivatives \( f^{(k)}(x_i) \) are approximated by linear combinations of function values at nearby grid points:
\begin{align*}
    f_i^{(k)} \approx \sum_{\alpha = i - m}^{i + m} c_\alpha f(x_\alpha).
\end{align*}
This leads to a discrete scheme where all derivatives are eliminated and only function values at grid points \( x_\alpha \) (\( \alpha = 0, 1, \ldots, J \)) appear.
\subsection{IMVL Scheme Construction for $v=w_x$-Type Problems}
To demonstrate the implementation of the IMVL, we begin with a simplified model where \( v = w_x \). This setting serves as a foundational case that illustrates the essential principles and discretization steps involved in the method. The resulting numerical scheme establishes a basis for extending the IMVL framework to more general and complex problems.
\par
With $v = w_x$, the IMVL is directly applied to both sides to obtain the following
\begin{align}
    \mathop{\int}_{xx} v{\div} \underset{xx}{\int}1 =\mathop{\int}_{xx}w_x{\div} \underset{xx}{\int}1. \label{eq-eq}
\end{align}
Our objective is to construct schemes attaining at least fourth-order accuracy, assuming that \( w(x) \) satisfies the required smoothness conditions. By Lemma~\ref{lem-2.1}, we have
                 %变限积分0,1,2阶导数
\begin{align}
    \underset{xx}{\int}v\div \underset{xx}{\int}1&=\sum_{k=0}^{3}\tilde{e}_kv_i^{(k)}+ {R_1},\label{eq-R1}\\
    \underset{xx}{\int}w_{x}\div \underset{xx}{\int}1&=\sum_{k=0}^{3}\tilde{e}_kw_i^{(k+1)}+ {R_2},\label{eq-R2}
\end{align}
where  $ \left |{R_j}\right |\leq C\tilde{B}_4, j=1, 2$, with \( C \) being a positive constant. 
\par
Differing from previous works where the integral parameters \(\varepsilon_1\) and \(\varepsilon_2\) were fixed, this study sets \(\varepsilon_1 = \varepsilon_2 = m h\) with \( m \in \mathbb{R} \). Under this setting, the coefficients in \eqref{eq-R1} and \eqref{eq-R2} are given by
\begin{align*}
    \tilde{e}_0&=1, \quad \tilde{e}_1=0,\quad 
    \tilde{e}_2=\frac{m^2}{12}h^2, \quad
    \tilde{e}_3=0, \quad
    \tilde{B}_4=\frac{m^4}{360}h^4. 
\end{align*}
Then, the derivatives in \eqref{eq-R1} and \eqref{eq-R2} need to be approximated by linear combinations of function values at discrete points. However,the choice of these points depends on the relative locations of variables \( v \) and \( w \) within the computational grid. 
To effectively capture the structural features of the governing equations, this paper proposes two discretization strategies, tailored to node-centered and staggered grid arrangements, respectively.
\subsubsection{Node-Centered Grid IMVL Scheme} 
We first consider the case where both $w(x)$ and $v(x)$ are defined on node-centered grids. The grid configuration is illustrated in Figure \ref{fig_1}.
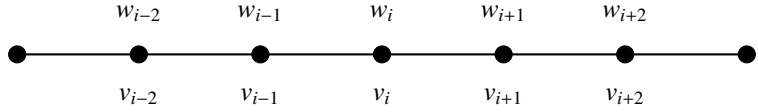
\begin{figure}[ h!]
\centering
\begin{tikzpicture}[scale=1. 6]
    % Draw line
    \draw[thick] (0,0) -- (6,0);

    % Draw filled circles
    \foreach \x in {0,1,2,3,4,5,6} {
        \filldraw (\x,0) circle (2pt);
    }
    % Labels

    \node[above] at (1,0.2) {$w_{i-2}$};
    \node[above] at (2,0.2) {$w_{i-1}$};
    \node[above] at (3,0.2) {$w_{i}$};
    \node[above] at (4,0.2) {$w_{i+1}$};
    \node[above] at (5,0.2) {$w_{i+2}$};

    \node[below] at (1,-0.2) {$v_{i-2}$};
    \node[below] at (2,-0.2) {$v_{i-1}$};
    \node[below] at (3,-0.2) {$v_{i}$};
    \node[below] at (4,-0.2) {$v_{i+1}$};
    \node[below] at (5,-0.2) {$v_{i+2}$};

\end{tikzpicture}
\caption{Node-centered grid with variables \( w \) and \( v \)}\label{fig_1}
\end{figure}
\par
We consider uniformly spaced nodes \( \{x_i\}_{i\in\mathbb{Z}} \) with spacing \( h = x_{i+1} - x_i \). The functions \( w \) and \( v \) are evaluated at the grid points as \( w(x_i) \) and \( v(x_i) \), respectively. To construct discrete schemes, we define a set of stencil points:
\begin{align}
\{ x_{i + p} \mid p\in \{-2, -1, 0, 1, 2\} \}.
\end{align}
We construct linear combinations of \( v_{i+p} \) and \( w_{i+p} \) at stencil points to approximate the derivative terms in \eqref{eq-R1} and \eqref{eq-R2}:
\begin{equation}\label{eq-new}
\begin{aligned}
    &\sum_{k=0}^{3}\tilde{e}_kv_i^{(k)} = \sum_{ p} a_pv_{i+p} + \mathcal{O}(h^4), \\
    &\sum_{k=0}^{3}\tilde{e}_kw_i^{(k+1)} = \sum_{p} \frac{d_p}{h} w_{i+p} + \mathcal{O}(h^4),
\end{aligned}
\end{equation}
where the coefficients \(a_p\) and \(d_p\) are unknowns to be determined.
\par
To determine the coefficients \(a_p\) and \(d_p\), we separately expand the right-hand sides of the two expressions in equation \eqref{eq-new} in Taylor series about the reference point \(x_i\). The corresponding Taylor expansions are
\begin{align*}
v_{i+p} = \sum_{k=0}^{\infty} \frac{(ph)^k}{k!} v^{(k)}(x_i), \quad
w_{i+p} = \sum_{k=0}^{\infty} \frac{(ph)^k}{k!} w^{(k)}(x_i).
\end{align*}
Then, the right-hand sides of \eqref{eq-new} become linear combinations of derivatives at the point \( x_i \). To ensure that the approximations achieve the accuracy required by the IMVL framework, we equate the coefficients of the corresponding derivative terms \( v^{(k)}(x_i) \) and \( w^{(k)}(x_i) \) with the IMVL coefficients \(\tilde{e}_k\) for \( k = 0,1,2,3 \), thereby forming a system of algebraic equations for the coefficients \( a_p \) and \( d_p \). 
\par
By solving the linear system, the coefficients are given as follows
\begin{gather*}
     a_{-2}=a_{2},\quad  a_{-1}=a_{1}= \frac{1}{12} (-48a_2+m^2), \quad a_{0}=\frac{1}{6} (6+36a_2-m^2),
      \\d_1=-d_{-1}= \frac{8-m^2}{12 },\quad d_0=0,\quad d_{2}=-d_{-2} =\frac{-2+m^2}{24}.
\end{gather*}
These coefficient expressions naturally satisfy $\sum_{i=-2}^{2} a_i = 1$ and $\sum_{i=-2}^{2} d_i = 0$.  Here, \(m\) and \(a_2\) are parameters that determine the coefficients \(a_p\) and \(d_p\). 
\par
For notational simplicity, we denote these parameter-dependent linear combinations, defined by the coefficients, as the operators \(\mathcal{A}(m,a_2)\) and \(\mathcal{H}(m)\):
\begin{align*}
    &\mathcal{A}(m,a_2)v_{i }:= a_{2} v_{i-2 }+ a_{1} v_{i-1}+a_0v_{i}+a_1 v_{i+1}+ a_2v_{i+2}, \quad
    \mathcal{H} (m)w_{i}:=\frac{1}{h}(d_{2}w_{i+2}+d_{1}w_{i+1}-d_{1}w_{i-1}-d_{2}w_{i-2}).
\end{align*}
From \eqref{eq-eq} and Taylor expansion, we obtain:
\begin{align}
  & \mathcal{A}(m,a_2)v_{i} = \mathcal{H} (m )w_{i}+ h^4 w_{i}^{(5)} (\frac{1}{30} + a_2  - \frac{{m }^2}{72}) 
    + h^6 w_{i}^{(7)}\left(\frac{15+630a_2  - 7 {m }^2}{3780}\right) + h^8w_{i}^{(9)}\left(\frac{3024a_2  + {m }^2}{241920}\right) +O(h^9). \label{eq-opp} 
\end{align}
This completes the numerical scheme construction for \( v = w_x \) on node-centered grids. A detailed truncation error analysis for the parameterized scheme introduced in \eqref{eq-opp} is provided in Appendix A.
As seen from the truncation error, choosing \( a_2 = \frac{m^2}{72} - \frac{1}{30} \) eliminates the fourth-order error, yielding a sixth-order accurate one-parameter scheme family.
\subsubsection{Staggered-Grid IMVL Scheme} 
We now consider the case where \( w \) and \( v \) are defined on a staggered grid, as illustrated in Figure  \ref{fig_2}. The role of \( \lambda \) is to represent the shift between variable positions arising from the staggered mesh structure.
\par
\begin{figure}[h!]
\centering
\begin{tikzpicture}[scale=1.6]

    % Draw main line
    \draw[thick] (0,0) -- (6,0);

    % Grid nodes for w_i (filled circles)
    \foreach \x in {0,1,2,3,4,5,6} {
        \filldraw[black] (\x,0) circle (2pt);
    }

    % Midpoints for v_{i+1/2} (open circles)
    \foreach \x in {0.5,1.5,2.5,3.5,4.5,5.5} {
        \draw[black, fill=white] (\x,0) circle (2pt);
    }

    % Labels for w_i (above)
    \node[below] at (1,-0.2) {$v_{i-2+\lambda}$};
    \node[below] at (2,-0.2) {$v_{i-1+\lambda}$};
    \node[below] at (3,-0.2) {$v_{i+\lambda}$};
    \node[below] at (4,-0.2) {$v_{i+1+\lambda}$};
    \node[below] at (5,-0.2) {$v_{i+2+\lambda}$}; 

    % Labels for v_{i+1/2} (below)
    \node[above] at (1.5,0.2) {$w_{i-\frac{3}{2}+\lambda}$};
    \node[above] at (2.5,0.2) {$w_{i-\f+\lambda}$};
    \node[above] at (3.5,0.2) {$w_{i+\f+\lambda}$};
    \node[above] at (4.5,0.2) {$w_{i+\frac{3}{2}+\lambda}$};

\end{tikzpicture}
\caption{Staggered grid with variables \( w\) and \( v \), with \( \lambda \in \left\{ 0, \tfrac{1}{2} \right\} \)}
\label{fig_2}
\end{figure}
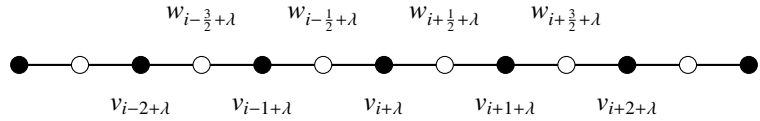
Assume that \( \{x_i\}_{i \in \mathbb{Z}} \) is a uniform grid with constant spacing \( h = x_{i+1} - x_i \), and define the midpoint by  $x_{i+\frac{1}{2}} = (x_i + x_{i+1})/{2}$ .
\par
In the staggered-grid configuration, \(v\) and \(w\) are located at different sets of grid points. Accordingly, two stencil point sets are introduced to support the discretization.
\[
\{ v_{i + p} \mid p\in \{-2, -1, 0, 1, 2\} \}, \quad
\{ w_{i + q} \mid q \in \{-\tfrac{3}{2}, -\tfrac{1}{2}, \tfrac{1}{2}, \tfrac{3}{2}\} \}.
\]
We follow the same construction strategy, with the only difference lying in the stencil point configuration. The approximations take the form:
\begin{equation}\label{eq-old}
\begin{aligned}
    &\sum_{k=0}^{3}\tilde{e}_kv_i^{(k)} = \sum_{p} a_{p} v_{i+p} + \mathcal{O}(h^4), \\
    &\sum_{k=0}^{3}\tilde{e}_kw_i^{(k+1)} = \sum_{q} \frac{b_{q+\f}}{h} w_{i+q} + \mathcal{O}(h^4),
\end{aligned}
\end{equation}
where the unknowns \( a_p \) and \( b_{q+\frac{1}{2}} \) are to be determined.
\par
To determine them, we apply Taylor expansions of \( v_{i+p} \) and \( w_{i+q} \) about the point \( x_i \) and substitute the results into \eqref{eq-old}. This yields linear combinations of derivatives, whose coefficients are then matched to the IMVL coefficients \( \tilde{e}_k \) for \( k = 0, 1, 2, 3 \), resulting in a linear system for \( a_p \) and \( b_{q+\frac{1}{2}} \).
\par
Solving the linear system yields the coefficients below. Since the coefficients \( a_p \) are identical to those in the node-centered formulation, they are omitted here.
\begin{gather*}
    b_1=-b_{0}= \frac{9-2m^2}{8}, \quad b_{2}=-b_{-1} =\frac{-1+2m^2}{24}.
\end{gather*}
The corresponding linear combinations are expressed as the following operators, where \( \lambda \in \left\{ 0, \tfrac{1}{2} \right\} \) indicates the staggered position of the variables.
\begin{align*}
    &\mathcal{A}(m,a_2)v_{i+\lambda }:= a_{2} v_{i-2+\lambda }+ a_{1} v_{i-1+\lambda}+a_0v_{i+\lambda}+a_1 v_{i+1+\lambda}+ a_2v_{i+2+\lambda},\\   
    & \delta(m)w_{i+\lambda}:=\frac{1}{h}(b_2w_{i+\frac{3}{2}+\lambda}+b_1w_{i+\f+\lambda}-b_1w_{i-\f+\lambda}-b_{2}w_{i-\frac{3}{2}+\lambda}).
\end{align*}
\begin{remark}
    The parameter \( \lambda \) reflects the staggered grid configuration of \( v \) and \( w \), acting as an index offset. This has no impact on the computation of \( a_{p} \) and \( b_{q+\f} \), as they are defined independently of \( \lambda \).
\end{remark}
To differentiate these parameters from those in \eqref{eq-opp}, \( m \) and \( a_2 \) are replaced by \( m^* \) and \( a_2^* \), respectively. This change is only for clarity and does not imply a fixed association with the numerical schemes introduced later. Applying equation \eqref{eq-eq} and Taylor expansion, we obtain
\begin{align}
  &\mathcal{A}(m^*,a^*_2)v_{i+\lambda} = \delta(m^*)w_{i+\lambda}+ h^4 w_{i+\lambda}^{(5)}( \frac{3}{640}  - \frac{{m^{*}}^2}{288} + a^*_2) \notag\\
   & + h^6 w_{i+\lambda}^{(7)}\left(\frac{135 + 80640 a^*_2 - 161 {m^{*}}^2 }{483840}\right) + h^8w_{i+\lambda}^{(9)}\left(\frac{3024a^*_2 + {m^{*}}^2}{241920}\right) +O(h^9).\label{eq-op} 
\end{align}
Appendix~A derives the truncation error for the above parameterized scheme. 
The truncation error shows that choosing
$a_2^* = \frac{{m^{*}}^2}{288} - \frac{3}{640}$
eliminates the leading fourth-order error, resulting in a sixth-order accurate scheme and forming a one-parameter family of sixth-order schemes.
\subsubsection{Examples of IMVL Schemes} 
Starting from the parameterized scheme, we select particular parameter values to derive specific numerical schemes, which will be employed in the numerical experiments presented later.
\par
For $ m=m^*=\frac{\sqrt 2}{2}, a_{2}=a_2^*=0$, the operators  $\mathcal{A}(m^*,{a_2}^*)$, $\delta({m}^*)$, $\mathcal{A}(m,a_2)$ and $\mathcal{H}(m) $ can be represented by the following schemes,where $\lambda \in \left\{ 0, \frac{1}{2} \right\}$
\begin{equation*}
 \textbf{HOS1:}\quad
\begin{aligned}
& \frac{1}{24}(v_{i+1+\lambda}+22v_{i+\lambda}+v_{i-1+\lambda}) = \frac{1}{h}(w_{i+\frac{1}{2}+\lambda} - w_{i-\frac{1}{2}+\lambda}) + O(h^4),\\
&\frac{1}{24}(v_{i+1}+22v_{i}+v_{i-1}) = \frac{1}{16h}(w_{i-2} -10 w_{i-1}+10 w_{i+1}-w_{i+2})+ O(h^4).
\end{aligned}
\end{equation*}
This numerical format adopted in \cite{shi2022fourth} was applied to solve nonlinear contaminant transport with adsorption.
\par
When $ m=m^*=\sqrt{2}, a_{2}=a_{2}^*=0$, we obtain new formats as follows:
\begin{equation*}
\textbf{HOS2:}\quad
\begin{aligned}
&\frac{1}{6}(v_{i-1+\lambda}+4v_{+\lambda}+v_{i+1+\lambda})=\frac{1}{8h}(-5w_{i-\frac{1}{2}+\lambda}+5w_{i+\frac{1}{2}+\lambda}+w_{i+\frac{3}{2}+\lambda}-w_{i-\frac{3}{2}+\lambda}) +O(h^4),\\
&\frac{1}{6}(v_{i-1}+4v_i+v_{i+1})=\frac{1}{2h}(w_{i+1}-w_{i-1}) +O(h^4).
\end{aligned}
\end{equation*}
\par
When $ m=m^*=\frac{\sqrt{11}}{2}, a_{2}=a_{2}^*=\frac{7}{1440}$, we obtain new sixth-order compact schemes as follows:
\begin{equation*}
\textbf{HOS3:}\quad
\begin{aligned}
&\frac{7}{1440}v_{i-2+\lambda}+\frac{151}{720}v_{i-1+\lambda}+\frac{137}{240}v_{i+\lambda}+\frac{151}{720}v_{i+1+\lambda}+\frac{7}{1440}v_{i+2+\lambda}\\
=&\frac{1}{16h}(-7w_{i-\frac{1}{2}+\lambda}+7w_{i+\frac{1}{2}+\lambda}+3w_{i+\frac{3}{2}+\lambda}-3w_{i-\frac{3}{2}+\lambda}) +O(h^6),\\
&\frac{7}{1440}v_{i-2}+\frac{151}{720}v_{i-1}+\frac{137}{240}v_i+\frac{151}{720}v_{i+1}+\frac{7}{1440}v_{i+2}
=\frac{1}{32h}(w_{i+2}-14w_{i-1}+14w_{i+1}-w_{i-2}) +O(h^6).
\end{aligned}
\end{equation*}
\par
Substituting the specific parameter values into the operators
$\mathcal{A}\left(9\sqrt{\frac{3}{119}}, \right.$ 
$\left.\frac{183}{76160}\right), \,
\delta\left(9\sqrt{\frac{3}{119}}\right), \,
\mathcal{A}\left(2\sqrt{\frac{6}{7}}, \frac{1}{70}\right)$ and
$\mathcal{H}\left(2\sqrt{\frac{6}{7}}\right)$,
we obtain the following eighth-order compact schemes
\begin{equation*}
\textbf{HOS4:}\quad
\begin{aligned}
&\frac{183}{76160}v_{i-2+\lambda}+\frac{3057}{19040}v_{i-1+\lambda}+\frac{3667}{5440}v_{i+\lambda}+\frac{3057}{19040}v_{i+1+\lambda}+\frac{183}{76160}v_{i+2+\lambda}\\
=&\frac{1}{2856h}(-1755w_{i-\frac{1}{2}+\lambda}+1755w_{i+\frac{1}{2}+\lambda}+367w_{i+\frac{3}{2}+\lambda}-367w_{i-\frac{3}{2}+\lambda}) +O(h^8),\\
&\frac{1}{70}v_{i-2}+\frac{8}{35}v_{i-1}+\frac{18}{35}v_i+\frac{8}{35}v_{i+1}+\frac{1}{70}v_{i+2}
=\frac{1}{84h}(5w_{i+2}+32w_{i+1}-32w_{i-1}-5w_{i-2}) +O(h^8).
\end{aligned}
\end{equation*}
\par
The above represents typical schemes, while various numerical formats can be derived by parameter variations, which will not be elaborated here.
\subsection{ High-Order Schemes for Contaminant Transport with Adsorption}
Building on the simplified setting, we now extend the IMVL framework to more complex models by developing high-order schemes for contaminant transport with adsorption.
\subsubsection{Periodic BC Numerical Schemes}
We first consider the construction of the scheme under periodic boundary conditions
\begin{equation}\label{eq-bc}
\begin{aligned}
& c_t+\phi(c)_t-(Dc_x)_x+(u c)_x =f(x, t), \quad(x, t) \in(0,x_R) \times(0, T], \\
&c(0,t)=c(x_R,t),\quad t\in (0, T],\\
&c(x,0)=c_0(x),\quad x\in (0,x_R).
\end{aligned}
\end{equation}
For the purpose of numerical scheme construction, we introduce the auxiliary variables \( z = -D c_x \), \( q = z_x \), and \( p = (u c)_x \), representing the diffusive flux, the divergence of the flux, and the convection term, respectively.
Accordingly, equation \eqref{eq-bc} can be equivalently written as:
\begin{equation}\label{rewrite}
\begin{cases}
    & c_t+\phi(c)_t+ q+p=f(x, t), \quad(x, t) \in(0,x_R) \times(0, T], \\
    & -\frac{z}{D}=c_x, \quad(x, t) \in(0,x_R) \times(0, T],\\
    & q=z_x,\quad(x, t) \in(0,x_R) \times(0, T],\\
    &p=(uc)_x,\quad(x, t) \in(0,x_R) \times(0, T].
\end{cases}
\end{equation}
%\begin{equation}\label{rewrite}
%\begin{cases}
%    & \left(uc+z \right)_x=c_t+\phi(c)_t+f(x, t), \quad(x, t) \in(0,x_R) \times(0, T], \\
%& c_x=-\frac{z}{D}, \quad(x, t) \in(0,x_R) \times(0, T]\\
%&c(0,t)=c(x_L,t),\quad t\in (0, T],\\
%&c(x,0)=c_0(x),\quad x\in (x_L,x_R).
%\end{cases}
%\end{equation}
\par
We adopt a block-centered grid as illustrated in Figure 3, where fluxes are defined at cell centers and variables at cell boundaries. The mesh is constructed by dividing the interval \( [0, x_R] \) into points \( \Pi:x_0 = 0 < x_1 < \cdots < x_{J-1} < x_J = x_R \).
\begin{figure}[h!] \label{fig333}
\centering
\begin{tikzpicture}[scale=1.3]
    % Draw line
    \draw[thick] (0,0) -- (8,0);
    
    % Draw open circles
    \draw (0,0) circle (3pt);
    \draw (2,0) circle (3pt);
    \draw (4,0) circle (3pt);
    \draw (6,0) circle (3pt);
    \draw (8,0) circle (3pt);

    % Draw filled circles
    \foreach \x in {1,3,5,7} {
        \filldraw (\x,0) circle (3pt);
    }
    
    % Labels
    \node[above] at (0,0.2) {$c_0$};
    \node[above] at (1,0.2) {$z_\frac{1}{2}$};
    \node[above] at (7,0.2) {$z_{J-\f}$};
    \node[above] at (8,0.2) {$c_{J}$};
    
    \node[below] at (0,-0.2) {$x_{0}$};
    \node[below] at (1,-0.2) {$x_\frac{1}{2}$};
    \node[below] at (2,-0.2) {$x_{1}$};
    \node[below] at (3,-0.2) {$x_{1+\f}$};
    \node[below] at (4,-0.2) {$\cdots$};
    \node[below] at (5,-0.2) {$\cdots$};
    \node[below] at (6,-0.2) {$x_{J-1}$};
    \node[below] at (7,-0.2) {$x_{J-\frac{1}{2}}$};
    \node[below] at (8,-0.2) {$x_{J}$};
    
    % Dots
    \node at (7,0) {$\cdots$};
\end{tikzpicture}
\caption{One-dimensional block-centered grids ( ◦-$c$, •-$z$)}
\end{figure}
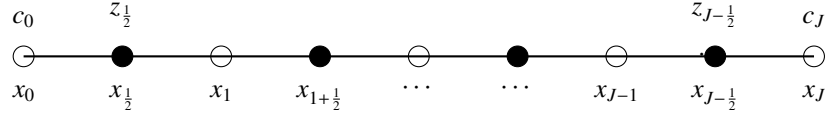
Let \( J \in \mathbb{N}_+ \) be a positive integer, and define the uniform grid size as \( h = x_R/J \). The grid nodes are given by \( x_i = i h \) for \( i = 0, \dots, J-1 \), and the intervals \( I_i = (x_i, x_{i+1}) \) form the primary mesh. The block-centered dual partition is defined as \( \Pi^* = \{ x_{i+\frac{1}{2}} \} \), where \( x_{i+\frac{1}{2}} = (x_i + x_{i+1}) / 2 \).

To prepare for spatial discretization, we analyze the structure of system~\eqref{rewrite}. The first equation is an ordinary differential equation involving a time derivative. The remaining three equations do not contain spatial derivatives on the left-hand side, while their right-hand sides are first-order derivatives with respect to \( x \). These are handled uniformly by considering them as \( v = w_x \)-type problems.
\par
Therefore, the spatial discretization employs IMVL schemes introduced in the previous section, while the temporal discretization adopts the following Euler scheme.
Let $N$ be a positive integer with $ \Delta t={T}  / {N}, t_n=n\Delta t$ and $c^n_i=c(x_i, t_n)$. The time difference operator is then defined as follows:
\begin{align*}
\partial_t c_i^{n+1}=\frac{c_i^{n+1}-c_i^{n}}{\Delta t}. 
\end{align*}
\par
Let $C_{i}$, $Z_{i+\f}$, $P_{i}$, $Q_{i}$ and $\Phi_i$ represent the approximations of $ c(x_i)$, $z(x_{i+\f})$, $p(x_i)$, $q(x_i)$ and $\phi(x_i)$ respectively. The full discretization scheme is expressed as
\begin{align}
& \partial_{t} C_{i}^{n+1}+ \partial_{t}\Phi_i^{n+1}+Q_{i}^{n+1} + P_i^{n} =  f_{i}^{n+1},\quad i= 1,2, \cdots, J, \\
&\A(m,a_2)\left(\frac{Z}{D}\right)_{i+1 / 2}^{n+1}  = -\delta(m) C_{i+1 / 2}^{n+1}, \quad i= 0,1, \cdots, J-1,\\
&\A(m,a_2) Q_i^{n+1}=\delta(m)  Z_i^{n+1},\quad i= 1,2, \cdots, J,\\
&\A(m^*,a_{2}^*) P_i^{n}=\mathcal{H} (m^*) (uC)_i^{n},\quad i= 1,2, \cdots, J.
\end{align}
along with the boundary and initial conditions
\begin{align}
&C_{-1}^{n} = C_{J-1}^{n}, \quad C_{0}^{n}  = C_{J}^{n}, \quad C_{1}^{n} = C_{J+1}^{n}, \quad C_{2}^{n} = C_{J+2}^{n}\label{boundary-1} \\
&P_{-1}^{n} = P_{J-1}^{n}, \quad P_{0}^{n}  = P_{J}^{n}, \quad Q_{-1}^{n} = Q_{J-1}^{n}, \quad Q_{0}^{n}  = Q_{J}^{n} \\
&Z_{1 / 2}^{n}  = Z_{J+1 / 2}^{n}, \quad Z_{-1 / 2}^{n}  = Z_{J-1 / 2}^{n} \label{boundary-2}\\
&C_{i}^{0}  = c_{0}\left(x_{i}\right)\label{eq-int}
\end{align}
Owing to the invertibility of the operator under periodic boundary conditions, the discrete scheme reduces to the following form
\begin{align}
&\A(m,a_2) \partial_{t} C_{i}^{n+1}+\A(m,a_2) \partial_{t} \Phi_{i}^{n+1}+\delta(m) Z_{i}^{n+1}  \notag \\
&= -\A(m,a_2)\A(m^*,a_2^*)^{-1}\mathcal{H} (m^*)(u C)_{i}^{n}+\A(m,a_2) f_{i}^{n+1},\quad i= 1,2, \cdots, J, \label{eq-dis1} \\
&\delta(m) C_{i+1 / 2}^{n+1}+\A(m,a_2)\left(\frac{Z}{D}\right)_{i+1 / 2}^{n+1}  = 0,  \quad i= 0,1, \cdots, J-1.\label{eq-dis2}
\end{align}
The fully discrete scheme involves four tunable parameters: \(m\), \(a_2\), \(m^*\), and \(a_2^*\) that determine the accuracy of the solution and flux approximations. By adjusting these parameters, fourth-, sixth-, and eighth-order schemes are obtained within a unified framework. 
The previously introduced HOS1–HOS4 schemes arise as special cases of this parameterized family.
\subsubsection{Dirichlet BC Numerical Schemes}
After periodic boundary conditions, we now focus on contaminant transport with adsorption under Dirichlet boundary conditions, governed by:
\begin{equation}
\begin{aligned}
& c_t+\phi(c)_t + z_x+(u c)_x =f(x, t), \quad(x, t) \in(x_L, x_R) \times(0, T], \\
&c_x=\frac{z}{D}, \quad(x, t) \in(x_L, x_R) \times(0, T], \\
&c(x_L,t)=g_L(t),\quad c(x_R,t)=g_R(t),\quad t\in (0, T],\\
&c(x,0)=c_0(x),\quad x\in (x_L,x_R).
\end{aligned}
\end{equation}
where \( g_L(t) \) and \( g_R(t) \) are prescribed boundary concentration functions at the left and right boundaries, respectively.
\par
We construct boundary-matching schemes HOS1-D and HOS2-D corresponding to the classical compact difference schemes HOS1 and HOS2. To maintain structural consistency, we retain the formulation \(v = w_x\) as the foundation of the scheme, modifying only the stencil near boundaries to ensure compatibility with Dirichlet conditions. Derivations follow below.\\
\textbf{(a):HOS1-D}
\par
Symmetric stencils become invalid near boundaries due to the unavailability of required neighboring points. For instance, \({\A} v_{1/2}\) cannot be approximated using its left neighbor. Although \(w_0\) is accessible in the evaluation of \({\mathcal{H}} w_1\), stencil adaptation remains necessary to satisfy Dirichlet conditions.

To this end, we introduce asymmetric stencils. Let $ \left\{ x_{\Lambda_*} \mid  \Lambda_* = \left\{ \frac{1}{2}, \frac{3}{2}, \frac{5}{2}, \frac{7}{2} \right\} \right\}$
 and $ \left\{ x_{\Lambda} \mid  \Lambda = \left\{0,1,2,3,4\right\} \right\}$
 denote the sets of stencil points for \(v\) and \(w\), respectively. Derivatives are approximated by linear combinations of function values over these stencils, given by
\begin{align*}
    \sum_{k=0}^{3}\tilde{e}_kv_i^{(k)} =\sum_{r\in \Lambda_* }l_{r+\f} v_{{r}} + O(h^4),\quad
    \sum_{k=0}^{3}\tilde{e}_kw_i^{(k+1)} = \sum_{r\in\Lambda }\frac{k_{r} }{h}w_{r}+ O(h^4).
\end{align*}
where the coefficients \(l_{r+\frac{1}{2}}\) and \(k_r\) are unknowns to be determined.
\par
Specifically, we perform Taylor expansions of \( v_{3/2}, v_{5/2}, v_{7/2} \) around \( v_{1/2} \), and of \( w_0, w_2, w_3, w_4 \) around \( w_1 \).These expansions convert these into linear combinations of derivatives, which are then matched to the coefficients \(\tilde{e}_k\) for \(k=0,1,2,3\). This process leads to a system of algebraic equations for the coefficients \( l_{r+1/2} \) and \( k_r \).
\par
Solving the system of linear equations, we have
\begin{align*}
    &l_{1}=\frac{(6+m^2)}{6},l_2=-\frac{5m^2}{12},\quad l_3=\frac{ m^2}{3},\quad l_4=-\frac{m^2}{12},\\
     k_{0}=\frac{-2 - m^2}{8},&\quad k_{1}=\frac{5 (-2 + m^2)}{12},\quad k_{2}=\frac{3 - m^2}{2},\quad  k_{3}=\frac{-2 + m^2}{4},\quad k_{4}=\frac{2 - m^2}{24}.
\end{align*}
Varying the parameter $m$ yields a family of numerical schemes for the left boundary.  The right boundary is treated analogously, but only the left boundary treatment is detailed here for conciseness.
\par
In HOS1-D, the parameter \( m = \frac{\sqrt{2}}{2} \) remains consistent with HOS1. The operators \(\hat{\A}\) and \(\hat{\mathcal{H}}\), incorporating boundary point information while preserving HOS1's core structure, are defined as follows, with \(\lambda \in \{0, \frac{1}{2}\}\).
\begin{align*}
    &\mathcal{A}v_{i+\f}=\frac{1}{24}(v_{i+\frac{3}{2} }+22v_{i+\frac{1}{2} }+v_{i-\frac{1}{2}}), \quad
    \delta w_{i+\lambda}=\frac{1}{h}(w_{i+\frac{1}{2}+\lambda} - w_{i-\frac{1}{2}+\lambda}),\\
    &\mathcal{H} w_{i}=\frac{1}{16h}(w_{i-2 }-10w_{i-1 }+10w_{i+1}-w_{i+2}),
\end{align*}
\begin{equation} 
\begin{aligned}
\hat{\A}v_{i+\f}=&\begin{cases}
\mathcal{A}v_{i+\f},\quad  \quad i = 1,2,\dots ,J-2,\\ 
 \frac{1}{24} (26 v_\frac{1}{2}-5 v_\frac{3}{2}+4 v_\frac{5}{2}-v_\frac{7}{2}),\quad i = 0,\\ 
\frac{1}{24} (26 v_{J-\frac{1}{2}}-5 v_{J-\frac{3}{2}}+4 v_{J-\frac{5}{2}}-v_{J-\frac{7}{2}}),\quad i = J-1,
\end{cases}\\
\hat{\mathcal{H} }w_{i}=&\begin{cases}
 \mathcal{H} w_{i},\quad  \quad i = 2,3,\dots ,J-2,\\ 
 \frac{1}{16h} (-5 w_0-10 w_1+20 w_2-6w_3+w_4), \quad i = 1,\\ 
\frac{1}{16h} (5 w_{J}+10 w_{J-1}-20 w_{J-2}+6w_{J-3}-w_{J-4}),\quad i = J-1.\notag
\end{cases}
\end{aligned}    
\end{equation}
This numerical format takes the following form
\begin{align*}
&\A \partial_t C_i^{n+1}+\A \partial_t \Phi_i^{n+1}+\delta Z_i^{n+1}=-\hat{\mathcal{H} }\left(u C\right)_i^n+\A f_i^{n+1}, i=1,2, \cdots, J-1,\\
& \delta C_{i+1 / 2}^{n+1}+\hat{\A}\left(\frac{Z}{D}\right)_{i+1 / 2}^{n+1}=0, \quad i=0,1, \cdots, J-1, \\
& C_0^{n+1}=c\left(x_L, t^{n+1}\right),\quad C_J^{n+1}=c\left(x_R, t^{n+1}\right).
\end{align*}
\textbf{(b):HOS2-D}
\par
The boundary treatment of HOS2-D builds on the HOS1-D approach. Specifically, \(\hat{\A}v_{1/2}\) uses the same asymmetric stencil as HOS1-D, while \(\hat{\delta}w_{1/2}\) is constructed from a five-point stencil
 $ \left\{ x_{\mathcal{J}} \mid  \mathcal{J} = \left\{0,1,2,3,4\right\} \right\}$, satisfying
\begin{align*}
     \underset{xx}{\int}w_x\div \underset{xx}{\int}1= \sum_{r\in\mathcal{J} }\frac{g_{r} }{h}w_{r}+ O(h^4).
\end{align*}
The coefficients \(g_r\) are obtained by Taylor expanding around \(w_{1/2}\) on \(\{ w_j \}_{j=0}^4\) and matching derivative coefficients with \(\tilde{e}_k\) (\(k=0,1,2,3\)). Solving the system of linear equations, we have
\begin{gather*}
  g_0=\frac{-11 - 2 m^2}{12},\quad g_1=\frac{17 + 14 m^2}{24},\\
  g_2=-\frac{3 (-1 + 2 m^2)}{8},\quad g_3=\frac{5 (-1 + 2 m^2)}{24},\quad g_4=\frac{1 - 2 m^2}{24}.
\end{gather*}
\par
In HOS2-D, the parameter \( m = \sqrt{2} \) remains consistent with HOS2. The operators \(\hat{\mathcal{A}}\) and \(\hat{\delta}\) incorporate boundary information while preserving the core structure of HOS2. The operators are defined as follows, with \(\lambda \in \{0, \frac{1}{2}\}\).
\begin{align*}
    &\mathcal{A}v_{i+\f+\lambda }=\frac{1}{6}(v_{i+\frac{3}{2}+\lambda}+4v_{i+\f+\lambda}+v_{i-\f+\lambda}), \quad
    \mathcal{H} w_{i}=\frac{1}{2h}(w_{i+1}-w_{i-1}),\\
    &\delta w_{i+\lambda}=\frac{1}{8h}(-5w_{i-\f+\lambda}+5w_{i+\f+\lambda}+w_{i+\frac{3}{2}+\lambda}-w_{i-\frac{3}{2}+\lambda}),
\end{align*}
\begin{equation}
\begin{aligned}
\hat{\A}v_{i+\f}&=\begin{cases}
\mathcal{A}v_{i+\f },\quad  \quad i = 1,2,\dots J-2,\\ 
 \frac{1}{6} (8v_{\f }-5 v_{\frac{3}{2}}+4 v_{\frac{5}{2}}-v_{\frac{7}{2}}),\quad i = 0,\\ 
\frac{1}{6} (8v_{J-\f}-5 v_{J-\frac{3}{2}}+4 v_{J-\frac{5}{2}}-v_{J-\frac{7}{2}}),\quad i = J-1,
\end{cases}\\
\hat{\delta}w_{i+\lambda}&=\begin{cases}
\delta w_{i+\lambda},\quad  \quad i = 1,2,\dots ,J-2,\\ 
 \frac{1}{8h} (-10 w_{\f-\lambda}+15 w_{\frac{3}{2}-\lambda}-9 w_{\frac{5}{2}-\lambda}+5w_{\frac{7}{2}-\lambda}-w_{\frac{9}{2}-\lambda}), \quad i = 0,\\ 
\frac{1}{8h} (10 w_{J-\f+\lambda}-15 w_{J-\frac{3}{2}+\lambda}+9 w_{J-\frac{5}{2}+\lambda}-5w_{J-\frac{7}{2}+\lambda}+w_{J-\frac{9}{2}+\lambda}),\quad i = J-1.\notag
\end{cases}
\end{aligned}
\end{equation}
This numerical format takes the following form
\begin{align}
&\A \partial_t C_i^{n+1}+\A \partial_t \Phi_i^{n+1}+\hat{\delta} Z_i^{n+1}=-\mathcal{H} \left(u C_i^n\right)+\A f_i^{n+1}, i=1,2, \cdots, J-1,\notag\\
& \hat{\delta }C_{i+1 / 2}^{n+1}+\hat{\A}\left(\frac{Z}{D}\right)_{i+1 / 2}^{n+1}=0, \quad i=0,1, \cdots, J-1, \notag\\
& C_0^{n+1}=c\left(x_L, t^{n+1}\right),\quad C_J^{n+1}=c\left(x_R, t^{n+1}\right).\notag
\end{align}
The construction of the numerical schemes is now complete.  
We next introduce the discrete inner products and norms essential for error and stability analysis.

\par
We define the discrete inner products and norms associated with the numerical scheme. Let 
$S_c = \{ C \mid \{C_i\}, \, x_i \in \Pi \}$ and $S_z = \{ Z \mid \{Z_{i+\frac{1}{2}}\}, \, x_{i+\frac{1}{2}} \in \Pi^* \}$ denote the spaces of periodic grid functions. For any \( P, Q \in S_c \) and \( U, V \in S_z \), the discrete inner products and corresponding norms are defined as:
\begin{align*}
 & \inner{P,Q}=\sum_{i=1}^{J} hP_iQ_i, \quad \n{P}=\inner{P,P}^{\f},\quad \n{P}_C=\underset{1\leq i\leq J}{\mathrm{max}}|P_{i}|, \\
 & (U,V)=\sum_{i=0}^{J-1} hU_{i+\f}V_{i+\f}, \quad \3{U}=(U,U)^{\f}, \quad
\3{U}_C=\underset{0\leq i\leq J-1}{\mathrm{max}}|U_{i+\f}|.
\end{align*}
\section{Mass Conservation, Stability and Convergence Analysis}
For the remainder of this paper, we restrict our analysis to numerical schemes satisfying the stability criterion \(a_0 - 2|a_1| - 2|a_2| > 0\) and consider those with periodic boundary conditions. The notations \(\mathcal{A}(m,a) \equiv \mathcal{A}\), \(\delta(m) \equiv \delta\), \(\mathcal{A}(m^*,a^*) \equiv \mathcal{A}^*\), and \(\mathcal{H}(m^*) \equiv \mathcal{H}\) are used throughout.
\subsection{Mass Conservation Analysis}
%             质量守恒的定理
\begin{theorem}
    The proposed scheme \eqref{eq-dis1}-\eqref{eq-dis2} ensures mass conservation, that is
\begin{align}\label{eq-mass}
\sum_{i=1}^{J}C_i^{n+1}h+\sum_{i=1}^{J}\Phi_i^{n+1}h =\sum_{i=1}^{J}C_i^{0}h+\sum_{i=1}^{J}\Phi_i^{0}h+\sum_{k=0}^{n}\sum_{i=1}^{J}f_i^{k+1}h\Delta t. 
\end{align}  
\end{theorem}
% %%%5                   开始证明定理4.1
\begin{proof} 
By multiplying \eqref{eq-dis1} by $ h$ and summing over $i=1,\dots ,J$, we obtain 
\begin{align}\label{eq-mass1}
   \sum_{i=1}^{J}\mathcal{A}\frac{( C^{n+1}_{i} - C^{n}_{i})h}{\Delta t} + \sum_{i=1}^{J}\mathcal{A}\frac{( \Phi^{n+1}_{i} -  \Phi^{n}_{i})h}{\Delta t}
    + \sum_{i=1}^{J}\delta Z_i^{n+1}h
    = -\sum_{i=1}^{J}-\A {\A^*}^{-1}\mathcal{H} (uC)_i^{n}h+\sum_{i=1}^{J}\mathcal{A} f_{i}^{n+1}h. 
\end{align}
Under periodic boundary conditions, we have
\begin{align}\label{eq-mass_2}
   &\sum_{i=1}^{J}\delta Z_i^{n+1}=\sum_{i=1}^{J}\frac{1}{h}(
    b_2Z^{n+1}_{i+\frac{3}{2}}+b_1Z^{n+1}_{i+\frac{1}{2}}-b_1Z^{n+1}_{i-\frac{1}{2}}-b_2Z^{n+1}_{i-\frac{3}{2}})\notag\\
   & =\frac{1}{h}\left(-(b_1+b_2)Z^{n+1}_{\frac{1}{2}}-b_2Z^{n+1}_{-\frac{1}{2}}-b_2Z^{n+1}_{\frac{3}{2}}+b_2Z^{n+1}_{J-\frac{1}{2}}+(b_1+b_2)Z^{n+1}_{J+\frac{1}{2}}+b_2Z^{n+1}_{J+\frac{3}{2}} \right)=0. 
\end{align}
and
\begin{align}\label{eq-mass_22}
    \sum_{i=1}^{J}\mathcal{H}  (uC)_i^{n}&=\sum_{i=1}^{J}\frac{1}{h}(
    d_2(uC)^{n}_{i+2}+d_1(uC)^{n}_{i+1}-d_1(uC)^{n}_{i-1}-d_2(uC)^{n}_{i-2})\notag\\
    &=\frac{1}{h}\left(-(d_1+d_2)(uC)^{n}_{0}-(d_1+d_2)(uC)^{n}_{1}-
    d_2(uC)^{n}_{2}-d_2(uC)^{n}_{-1}\right)\notag\\
   &\quad +\frac{1}{h}\left((d_1+d_2)(uC)^{n}_{J}+(d_1+d_2)(uC)^{n}_{J+1}+d_2(uC)^{n}_{J+2}+d_2(uC)^{n}_{J-1}\right)\notag\\
   &=0.
\end{align}
Since \(\mathcal{A}\) and \({\mathcal{A}^*}^{-1}\) are linear, their product \(\mathcal{A}{\mathcal{A}^*}^{-1}\) is also linear, and by \eqref{eq-mass_22} we obtain
\begin{align}\label{eq-mass_222}
\sum_{i=1}^J \A \A^{*-1} \mathcal{H}  (uC)_i^n = \A \A^{*-1} \left( \sum_{i=1}^J \mathcal{H}  (uC)_i^n \right) = 0.
\end{align}
For $ g\in {S_c}$, we have
\begin{align}\label{eq-mass_3}
    \sum_{i=1}^{J}\mathcal{A}g_i h&=\sum_{i=1}^{J}
    (a_2 g_{i-2}+ a_1 g_{i-1}+a_0 g_i+a_1 g_{i+1}+ a_2g_{i+2})h\notag\\
    &=(a_0+2a_1+2a_2)\sum_{i=1}^{J}g_{i}h =\sum_{i=1}^{J}g_{i}h.
\end{align}
By inserting \eqref{eq-mass_2} and \eqref{eq-mass_222} into \eqref{eq-mass1}, and using \eqref{eq-mass_3}, we obtain
\begin{align}\label{eq-mass_4}
   \sum_{i=1}^{J}\frac{C_i^{n+1}-C_i^{n}}{\Delta t}h +\sum_{i=1}^{J}\frac{\Phi_i^{n+1}-\Phi_i^{n}}{\Delta t}h =\sum_{i=1}^{J}f_i^{n+1}h. 
\end{align}
Summing \eqref{eq-mass_4} from  0 to $n$ and multiplying by $ \Delta t$ yields \eqref{eq-mass}, thus proving the result.
\end{proof}%引用公式前面也要加空格哦
\subsection{Stability Analysis}
To establish the stability of the IMVL scheme, we first introduce several lemmas that form the basis of the proof.
\begin{lemma}\label{lem-3.1}
    Let $ C\in S_c, Z\in S_z$, $ R_a \coloneqq a_0-2a_1-2a_2 $, $ R_b \coloneqq  |a_0|+2|a_1|+ 2|a_2| $, we have 
\begin{align*}
   & R_a\n{C}^2\leq \inner{\A{C},C}\leq R_b\n{C}^2,\quad  R_a^2\n{C}^2\leq \inner{\A{C},\A{C}}\leq R_b^2\n{C}^2, \\
   & R_a\3{Z}^2\leq (\A{Z}, Z)\leq  R_b\3{Z}^2 ,\quad 
     R_a^2\3{Z}^2\leq (\A{Z},\A{Z})\leq  R_b^2\3{Z}^2.
\end{align*}
\end{lemma}
\begin{proof}
The definition of  $ \A$ ( or $ \A^*$) and \eqref{boundary-1} imply that
\begin{align*}
     \inner{\A{C},C}&=\sum_{i=1}^{J}(a_2 C_{i-2}+ a_1 C_{i-1}+a_0 C_i+a_1 C_{i+1}+ a_2C_{i+2})C_{i}h \\
     &\geq h\sum_{i=1}^{J}\Bigl(a_0C_{i}^2-a_2(\frac{C_{i-2}^2}{2}+\frac{C_{i}^2}{2})-a_1(\frac{C_{i-1}^2}{2}+\frac{C_{i}^2}{2})
      -a_1(\frac{C_{i+1}^2}{2}+\frac{C_{i}^2}{2})-a_2(\frac{C_{i+2}^2}{2}+\frac{C_{i}^2}{2})\Bigr) \\
     &\geq h\sum_{i=1}^{J}(a_0-2a_1-2a_2)C_{i}^2= R_a\n{C}^2. 
\end{align*}
Applying the Cauchy-Schwarz inequality yields
\begin{align*}
   \inner{\A{C},C}
   &\leq h\Biggl(\left |a_0\sum_{i=1}^{J}C_{i}^2  \right |+\left |a_2\sum_{i=1}^{J}C_{i+2}C_{i}\right |+\left |a_1\sum_{i=1}^{J}C_{i+1}C_{i}\right |+\left |a_2\sum_{i=1}^{J}C_{i-2}C_{i}  \right |+\left |a_1\sum_{i=1}^{J}C_{i-1}C_{i}  \right |\Biggr) \\
    &\leq(|a_0|+2|a_1|+ 2|a_2|)\n{C}^2=R_b\n{C}^2. 
\end{align*}
The positive definite and self-adjoint operator $\A$ 
admits a square root  ($\A=\mathcal{B}_x\mathcal{B}_x$) with the same properties, yielding
\begin{align*}
  R_a^2\n{C}^2 \leq R_a \inner{ \mathcal{B}_x C,\mathcal{B}_x C} \leq \inner{\A C,\A C} \leq R_b \inner{\mathcal{B}_x C,\mathcal{B}_x C}\leq R_b^2\n{C}^2.
\end{align*}
The proofs for $ (\A{Z},Z)$ and $ (\A{Z}, \A Z)$ follow analogous reasoning.
\end{proof}
\begin{lemma}\label{lem-3.2}
    Let $ C\in S_c, Z\in S_z $, it follows that
\begin{align*}
    \inner{\delta Z,C}=-(Z,\delta C). 
\end{align*}
\end{lemma}    
\begin{proof}
By the definitions of inner products and periodic conditions, we obtain
\begin{align*}
    \inner{\delta Z,C}&=\sum_{i=1}^{J}\frac{1}{h}(b_2Z_{i+\frac{3}{2}}+b_1Z_{i+\frac{1}{2}}-b_1Z_{i-\frac{1}{2}}-b_2Z_{i-\frac{3}{2}})
   C_ih\\
    &=\sum_{i=0}^{J-1}b_2Z_{i+\frac{1}{2}}C_{i-1} + b_2( -Z_{\frac{1}{2}}C_{-1}- Z_{\frac{3}{2}}C_{0} +Z_{J+\frac{1}{2}}C_{J-1}+Z_{J+\frac{3}{2}}C_{J})-\sum_{i=0}^{J-1}b_1Z_{i+\frac{1}{2}}C_{i+1}.\\
    & \quad +\sum_{i=0}^{J-1}b_1Z_{i+\frac{1}{2}}C_{i}+b_1(-Z_{\frac{1}{2}}C_{0}+Z_{J+\frac{1}{2}}C_{J} )  
    -\sum_{i=0}^{J-1} b_2Z_{i+\frac{1}{2}}C_{i+2}-b_2(Z_{-\frac{1}{2}}C_{1} -Z_{J-\frac{1}{2}}C_{J+1})\\
    &=-\frac{1}{h}\sum_{i=0}^{J-1}(b_2C_{i+2}+b_1C_{i+1}-b_1C_{i}-b_2C_{i-1})Z_{i+\frac{1}{2}}h
    =-(Z,\delta C), 
 \end{align*}   
 which completes the proof. 
\end{proof}
\begin{lemma}\label{lem-3.3}
     Let $ C\in S_c$ and $Z\in S_z$, under the condition of periodic boundaries, the operators
     $\A $ and $\A^{-1}$ (or $\A^* $, ${\A^*}^{-1}$ )  are commutative and positive definite symmetric operators,
      We then have that
\begin{align*}
    \frac{1}{R_b^2}\n{C}^2\leq \n{\A^{-1}C}^2\leq \frac{1}{R_a^2}\n{C}^2, \quad \frac{1}{R_b^2}\3{Z}^2\leq \3{\A^{-1}Z}^2\leq \frac{1}{R_a^2}\3{Z}^2.
\end{align*}
\end{lemma}
\begin{proof}
Since the matrix representation of \(\mathcal{A}\) is symmetric, the result follows directly.
\end{proof}
\begin{lemma}\label{lem-3.4}
 Let $ g\in S_c$, $ R_d\coloneqq 4d_1^2+4d_2^2+8| d_1d_2 |$ , we have 
\begin{align*}
    \n{\mathcal{H} g}^2 \leq \frac{R_d}{h^2}\n{{g}}^2.
\end{align*}
\end{lemma}
\begin{proof}
Using the Cauchy-Schwarz inequality and \eqref{boundary-1}, it directly follows that
\begin{align*}
    &\n{\mathcal{H}  g}^2 =\inner{\mathcal{H} g,\mathcal{H} g}\\
    =&\sum_{i=1}^{J}\frac{1}{h^2}(d_2g_{i+2}+d_1g_{i+1}-d_1g_{i-1}-d_2g_{i-2})^2h\\
    \leq & \frac{1}{h^2}\left(\sum_{i=1}^{J}  \left(d_2^2g^2_{i+2}h+  d_1^2g^2_{i+1}h+ d_1^2g^2_{i-1}h+ d_2^2g^2_{i-2}h \right)
    + \left|\sum_{i=1}^{J} 2 d_1d_2 g_{i+2}g_{i+1}h\right|+\left|\sum_{i=1}^{J} 2d_1d_2g_{i+2}g_{i-1}h\right| \right.\\
    &\left. +2d_2^2\left|\sum_{i=1}^{J} g_{i+2}g_{i-2}h\right|+2d_1^2\left|\sum_{i=1}^{J} g_{i+1}g_{i-1}h\right|+ \left|\sum_{i=1}^{J} 2 d_1d_2g_{i+1}g_{i-2}h\right|+\left|\sum_{i=1}^{J}2d_1d_2 g_{i-1}g_{i-2}h\right|\right)\\
    \leq&\frac{1}{h^2}( 4d_1^2+4d_2^2+8| d_1d_2 |)\n{g}^2= \frac{R_d}{h^2}\n{g}^2.
 \end{align*} 
\end{proof}
\begin{lemma}[{\cite{eymard2006combined}}]\label{lem-yinyong}
     Let $\Upsilon(s) $, $s\in(0,x_R)$ be defined by
\begin{equation}
\Upsilon (s)=\psi(s) s-\int_0^s \psi(\tau) d \tau,
\end{equation}
with $\psi$ satisfying properties \eqref{eq-PP}. Then there exists a positive constant $K_*$ such that $\Upsilon (s) \geq \frac{1}{2} K_* s^2$.
\end{lemma}
\begin{theorem}\label{thm-3.2}
(Stability) Let  $ C$  and $Z$ satisfy \eqref{eq-dis1}-\eqref{eq-dis2}. Then, for  $\Delta t\leq \tau_0 $ and $ h\leq h_0 $ , where $\tau_0, h_0 $  and $ K$ are positive constants, the stability estimate is valid
\begin{align*}
    \n{C^{N}}^2+\3{Z^{N}}^2 \le K(\n{\Psi ^0}^2+\3{Z ^0}^2+\sum_{n=0}^{N} \Delta t\n{f^{n}}^2).
\end{align*}  
\end{theorem}
\begin{proof}
Taking the inner product of both sides of \eqref{eq-dis1} with \(\partial \mathcal{A}^{-1} C^{n+1}\), we obtain
\begin{align}\label{eq-34}
& \langle\partial_t\A C^{n+1}, \partial_t\A^{-1} C^{n+1}\rangle+\langle\partial_t\A \Phi^{n+1}, \partial_t\A^{-1} C^{n+1}\rangle+\langle\delta Z^{n+1}, \partial_t\A^{-1} C^{n+1}\rangle\notag \\
= & -\langle \A{\A^*}^{-1}\mathcal{H} (u C)^n, \partial_t\A^{-1} C^{n+1}\rangle+\langle\A f^{n+1}, \partial_t\A^{-1} C^{n+1}\rangle.
\end{align}
From equation \eqref{eq-34}, it can be readily concluded that
\begin{align}\label{eq-stability}
\n{ \partial_t C^{n+1}}^2+\langle\partial_t\Phi^{n+1}, \partial_t C^{n+1}\rangle+\langle\delta Z^{n+1}, \partial_t\A^{-1} C^{n+1}\rangle 
= -\langle {\A^*}^{-1}\mathcal{H} (u C)^n, \partial_t C^{n+1}\rangle+\langle f^{n+1}, \partial_t C^{n+1}\rangle.
\end{align}
Due to the monotonicity of \(\phi\), we have 
\begin{equation} \label{eq-sta-left1}
\langle\partial_t \Phi^{n+1}, \partial_t C^{n+1}\rangle=\frac{1}{\Delta t^2}\langle\phi\left(C^{n+1}\right)-\phi\left(C^n\right), C^{n+1}-C^n\rangle \geq 0 .
\end{equation}
By using \eqref{eq-dis2} and Lemma \ref{lem-3.2}, we calculate and organize to obtain
\begin{align}\label{eq-sta-left2}
& \langle\delta Z^{n+1}, \partial_t \A^{-1} C^{n+1}\rangle=-(Z^{n+1}, \partial_t \A^{-1} \delta C^{n+1}) =(Z^{n+1}, \partial_t(\frac{Z}{D})^{n+1})  \notag\\
= & \frac{1}{2 \Delta t}\left((Z^{n+1},\left(\frac{Z}{D}\right)^{n+1}) -(Z^n,\left(\frac{Z}{D}\right)^n) \right)+\frac{1}{2 \Delta t}(Z^{n+1}-Z^n, \frac{Z^{n+1}-Z^n}{D})  \notag\\
\geq & \frac{1}{2 \Delta t}\left((Z^{n+1},\left(\frac{Z}{D}\right)^{n+1}) -(Z^n,\left(\frac{Z}{D}\right)^n) \right) .
\end{align}
Multiplying \eqref{eq-stability} by $\Delta t$, sum over $N-1$ and using \eqref{eq-sta-left1},\eqref{eq-sta-left2}, we can  obtain
\begin{equation} \label{eq-sta-half}
\begin{aligned}
 \sum_{n=0}^{N-1} \Delta t\left\|\partial_t C^{n+1}\right\|^2+\frac{1}{2 D^*}\left(Z^N, Z^{N}\right) 
\leq  \frac{1}{2 D_*}\left(Z^0, Z^0\right)+\sum_{n=0}^{N-1} \Delta t\left|\langle \mathcal{H} (u C)^n, \partial_t {\A ^*}^{-1} C^{n+1}\rangle\right| +\sum_{n=0}^{N-1} \Delta t\left|\langle f^{n+1}, \partial_t C^{n+1}\rangle\right|.
\end{aligned}
\end{equation}
Applying Lemma \ref{lem-3.3} and \ref{lem-3.4}, the inequality $|\alpha \beta| \leq \varepsilon \alpha^2+\frac{1}{4 \varepsilon} \beta^2$ and define $L_u:=\max(u)$ we have 
\begin{align}
   & \Delta t\sum_{n=0}^{N-1}\left|\langle\mathcal{H} (u C)^n, \partial_t {\A^{*}} ^{-1} C^{n+1}\rangle\right| \leq \frac{L_u^2 R_d}{4\varepsilon h^2}\sum_{n=0}^{N-1}\Delta t\n{C^{n}}^2
    +\frac{\varepsilon}{R_a^2} \sum_{n=0}^{N-1}\Delta t\n{\pa C^{n+1}}^2\label{eq-sta-right1},\\
   &\Delta t\sum_{n=0}^{N-1}\left|\langle f^{n+1}, \partial_t C^{n+1}\rangle\right|\leq
   \frac{1}{4 \varepsilon} \sum_{n=0}^{N-1}\Delta t\n{f^{n+1}}^2+\varepsilon \sum_{n=0}^{N-1}\Delta t \n{\pa C^{n+1}}^2\label{eq-sta-right2}.
\end{align}
Substituting the estimates \eqref{eq-sta-right1},\eqref{eq-sta-right2} into \eqref{eq-sta-half}, 
\begin{align}\label{eq-41}
    (1-\frac{1+R_a^2}{R_a^2}\varepsilon)\sum_{n=0}^{N-1}\Delta t\n{C^{n+1}}^2
    + \frac{1}{2D^*}\3{Z^N}^2 \le \frac{1}{2D_*}\3{Z^0}^2+ \frac{L_u^2 R_d}{4\varepsilon h^2}\sum_{n=0}^{N-1}\Delta t\n{C^{n}}^2 +\frac{1}{4 \varepsilon} \sum_{n=0}^{N-1}\Delta t\n{f^{n+1}}^2
\end{align}
By setting \(\psi(c) = c + \phi(c)\) and defining \(\Psi_i^{n+1} = \psi(C_i^{n+1})\), \eqref{eq-dis1} can be equivalently expressed as
\begin{equation}\label{eq-sta-part2-1}
\A \partial_t \Psi_i^{n+1}+\delta Z_i^{n+1}=-\A{\A^*}^{-1}\mathcal{H} (u C)_i^n+\A f_i^{n+1}.
\end{equation}
Applying the inner product with \(\A^{-1} C^{n+1}\) to both sides of \eqref{eq-sta-part2-1} leads to
\begin{align*}
\langle\A \partial_t \Psi^{n+1}, \A^{-1} C^{n+1}\rangle+\langle\delta Z^{n+1}, \A^{-1} C^{n+1}\rangle
=-\langle\A{\A^*}^{-1}\mathcal{H} (u C)^n, \A^{-1} C^{n+1}\rangle+\langle\A f^{n+1}, \A^{-1} C^{n+1}\rangle.
\end{align*}
By the commutativity of the operators, equation~\eqref{eq-dis2}, and Lemma \ref{lem-3.2}, we obtain
\begin{equation}\label{eq-111}
\langle\partial_t \Psi^{n+1}, C^{n+1}\rangle+( Z^{n+1}, \left(\frac{Z}{D}\right)^{n+1})=-\langle\mathcal{H} ( u C)^n, {\A^*}^{-1}  C^{n+1}\rangle+\langle f^{n+1}, C^{n+1}\rangle.
\end{equation}
It follows from Lemma~\ref{lem-yinyong} that the following relation holds
\begin{equation}
\Upsilon \left(C_i^{n+1}\right)-\Upsilon \left(C_i^n\right)=\left(\psi\left(C_i^{n+1}\right)-\psi\left(C_i^n\right)\right) C_i^{n+1}-\int_{C_i^n}^{C_i^{n+1}}\left(\psi(\tau)-\psi\left(C_i^n\right)\right) d \tau,
\end{equation}
where the integrand is nonnegative due to $\psi$ being nondecreasing. It follows that
\begin{equation}\label{eq-Upsilon}
\Upsilon \left(C_i^{n+1}\right)-\Upsilon \left(C_i^n\right)\le\left(\psi\left(C_i^{n+1}\right)-\psi\left(C_i^n\right)\right) C_i^{n+1}.
\end{equation}
which yields
\begin{equation}\label{eq-sta11}
\langle\partial_t \Psi^{n+1}, C^{n+1}\rangle = \frac{1}{\Delta t}\inner{ \Psi^{n+1}- \Psi^{n},C^{n+1}}
\geq \frac{1}{\Delta t} \sum_{i=1}^J\left(\Upsilon \left(C_i^{n+1}\right)-\Upsilon \left(C_i^n\right)\right) h.
\end{equation}
Since $D_*<D<D^*$, the following result is obtained
\begin{align}\label{eq-sta22}
( Z^{n+1}, \left(\frac{Z}{D}\right)^{n+1})\ge \frac{1}{D^*}\3{Z^{n+1}}^2.
\end{align}
Multiplying \eqref{eq-111} by $\Delta t$, summing on $N-1$ and using \eqref{eq-sta11}, \eqref{eq-sta22} we obtain
\begin{align}\label{eq-sta33}
    \sum_{n=0}^{N-1}\sum_{i=1}^J\left(\Upsilon \left(C_i^{n+1}\right)-\Upsilon \left(C_i^n\right)\right) h + \frac{1}{D^*}\sum_{n=0}^{N-1}\3{Z^{n+1}}^2\Delta t 
    \leq \sum_{n=0}^{N-1}\left|\langle\mathcal{H} ( u C)^n, {\A^*}^{-1}  C^{n+1}\rangle\right|\Delta t + \sum_{n=0}^{N-1}\left|\langle f^{n+1}, C^{n+1}\rangle\right|\Delta t.
\end{align}
According to Lemma \ref{lem-yinyong}, we obtain
\begin{equation}\label{eq-sta44}
\frac{1}{2} K_* C^2 \leq \Upsilon (C) \leq \Psi(C) C \leq(\Psi(C))^2.
\end{equation}
Using \eqref{eq-sta44}, Lemma \ref{lem-3.4} and the inequality $|\alpha \beta| \leq \varepsilon \alpha^2+\frac{1}{4 \varepsilon} \beta^2$, we organize equation \eqref{eq-sta33}
\begin{align}\label{eq-sta-all2}
    \frac{1}{2}K_*\n{C^N}^2 \leq \n{\Psi ^0}^2+\frac{L_u^2 R_d}{4\varepsilon h^2}\sum_{n=0}^{N-1}\Delta t\n{C^{n}}^2
    +\frac{1+R_a^2}{R_a^2}\varepsilon \sum_{n=0}^{N-1}\Delta t\n{ C^{n+1}}^2
   +\frac{1}{D^*}\sum_{n=0}^{N-1}\3{Z^{n+1}}^2\Delta t+ \frac{1}{4 \varepsilon} \sum_{n=0}^{N-1}\Delta t\n{f^{n+1}}^2.
\end{align}
By summing equations~\eqref{eq-41} and~\eqref{eq-sta-all2}, and choosing \(\varepsilon\) sufficiently small such that
$1 - \frac{1 + R_a^2}{R_a^2} \varepsilon \ge 0$,
we can then apply Gronwall’s inequality to obtain the estimate:
\begin{equation}
\left\|C^{N}\right\|^2+\3{Z^{N}}^2 \leq K\left(\left\|\Psi^0\right\|^2+{\3{Z^0}}^2+\sum_{n=0}^N \Delta t\left\|f^{n+1}\right\|^2\right) .
\end{equation}
This completes the proof.
\end{proof}
\subsection{Convergence Analysis}
We now analyze convergence and error estimates for the scheme. By Taylor’s expansion, it follows that
\begin{align}
& \A \partial_t c_i^{n+1}+\A \partial_t \phi_i^{n+1}+\delta z_i^{n+1}=- \A{\A^*}^{-1}\mathcal{H} (u c)_i^n+\A f_i^{n+1}+R_i^{1, n+1},\label{eq-con1}\\
& \delta c_{i+1 / 2}^{n+1}+\mathcal{A}\left(\frac{z}{D}\right)_{i+1 / 2}^{n+1}=R_{i+1 / 2}^{2, n+1}, \label{eq-con2}
\end{align}
where $R_i^{1, n+1}=\mathcal{O}\left(h^s+\Delta t\right), R_{i+1 / 2}^{2, n+1}=\mathcal{O}\left(h^s\right)$ with $s=4,6,8$.
\par
Define $\eta=C-c$, $\xi=Z-z$.  Using  \eqref{eq-dis1}-\eqref{eq-dis2} and \eqref{eq-con1}- \eqref{eq-con2}, the error equations are obtained as follows
\begin{align}
& \A \partial_t \eta_i^{n+1}+\A \partial_t\left(\phi\left(C_i^{n+1}\right)-\phi\left(c_i^{n+1}\right)\right)+\delta \xi_i^{n+1}=-\A{\A^*}^{-1}\mathcal{H} (u \eta)_i^n+R_i^{1, n+1},\label{eq-con11} \\
& \delta \eta_{i+1 / 2}^{n+1}+\A\left(\frac{\xi}{D}\right)_{i+1 / 2}^{n+1}=R_{i+1 / 2}^{2, n+1}.\label{eq-con22}
\end{align}
\begin{theorem}
Assume $c,z\in C^2((0,T];C^{s+1}(0,x_R))$ are the exact solutions of \eqref{eq-bc} with s=4,6,8. Let C, Z be the solutions of the scheme given by equations \eqref{eq-dis1}–\eqref{eq-dis2}, and there exists a constant $K > 0$ such that the inequality below holds
\begin{equation}
\sum_{n=0}^{N} \Delta t\left\|(C-c)^{n+1}\right\|^2 \leq K\left(h^{2s}+\Delta t^2\right).
\end{equation}
\end{theorem}
\begin{proof}
    Let $C^0_i = c^0(x_i)$, which implies that $ \eta^0 = 0$, $ \phi(C^0)-\phi(c^0) = 0$. By multiplying equation  \eqref{eq-con11} by $\Delta t$ and summing on n, we obtain
\begin{equation}\label{eq-con-33}
\A \eta_i^{n+1}+\A\left(\phi\left(C_i^{n+1}\right)-\phi\left(c_i^{n+1}\right)\right)+\sum_{k=0}^n \Delta t \delta \xi_i^{k+1}=-\sum_{k=0}^n \Delta t\A{\A^*}^{-1}\mathcal{H} (u \eta)_i^k+\sum_{k=0}^n \Delta t R_i^{1, k+1}.
\end{equation}
By taking the inner product of \eqref{eq-con-33} with $ \A^{-1}\eta^{n+1}$, we get
\begin{align}\label{eq-con-all1}
& \langle\eta^{n+1}, \eta^{n+1}\rangle+\langle\phi\left(C^{n+1}\right)-\phi\left(c^{n+1}\right), \eta^{n+1}\rangle+\langle\sum_{k=0}^n \Delta t \delta \xi^{k+1},\A^{-1} \eta^{n+1}\rangle \notag\\
= & -\langle\sum_{k=0}^n \Delta t \A{\A^*}^{-1} \mathcal{H} (u \eta)^k,\A^{-1} \eta^{n+1}\rangle+\langle\sum_{k=0}^n \Delta t R^{1, k+1},\A^{-1} \eta^{n+1}\rangle.
\end{align}
By summing equation \eqref{eq-con-all1} from $ n=0$ to N, we obtain
\begin{align}\label{eq-com-all2}
& \sum_{n=0}^{N}\langle\eta^{n+1}, \eta^{n+1}\rangle+\sum_{n=0}^{N}\langle\phi\left(C^{n+1}\right)-\phi\left(c^{n+1}\right), \eta^{n+1}\rangle+\sum_{n=0}^{N}\langle\sum_{k=0}^n \Delta t \delta \xi^{k+1}, \A^{-1} \eta^{n+1}\rangle\notag \\
& =-\sum_{n=0}^{N}\langle\sum_{k=0}^n \Delta t \mathcal{H} (u \eta)^k, {\A^*}^{-1} \eta^{n+1}\rangle+\sum_{n=0}^{N}\langle\sum_{k=0}^n \Delta t R^{1, k+1}, \A^{-1} \eta^{n+1}\rangle.
\end{align}
We proceed to estimate the terms in equation~\eqref{eq-com-all2}, denoted by \( M_1, M_2, \ldots, M_5 \). Define $\beta^n=\sum_{k=0}^n\left\|\eta^{k+1}\right\|^2$, we can get
\begin{equation}
    M_1=\sum_{n=0}^{N}\langle\eta^{n+1}, \eta^{n+1}\rangle=\sum_{n=0}^{N}\left\|\eta^{n+1}\right\|^2=\beta^N.\label{eq-con-left1}
\end{equation}
Due to the monotonicity of $\phi$, it follows that
\begin{equation}\label{eq-con-g2}
M_2=\sum_{n=0}^{N}\langle\phi\left(C^{n+1}\right)-\phi\left(c^{n+1}\right), C^{n+1}-c^{n+1}\rangle \geq 0.
\end{equation}
We define $\rho^n=\frac{\Delta t}{2 D^*}\3{\sum_{k=0}^n \xi^{k+1}}^2 $.
Using the fundamental equality
\begin{equation}\label{eq-inequality}
2 \sum_{n=1}^N\left(\frac{1}{D} a^n, \sum_{k=1}^n a^k\right)=\left(\sum_{n=1}^N a^n, \frac{1}{D} \sum_{n=1}^N a^n\right)+\sum_{n=1}^N\left(a^n, \frac{1}{D} a^n\right)
\end{equation}
along with \eqref{eq-con22}, Lemmas~\ref{lem-3.2} and~\ref{lem-3.3},  
we derive the following estimate for \(M_3\):
\begin{align}\label{eq-con-g3}
M_3=&\sum_{n=0}^N\langle\sum_{k=0}^n \Delta t \delta \xi^{k+1}, \mathcal{A}^{-1} \eta^{n+1}\rangle\notag\\
=&\Delta t \sum_{n=0}^{N}\left(\sum_{k=0}^n \xi^{k+1}, \frac{1}{D} \xi^{n+1}\right)-\Delta t \sum_{n=0}^{N}\left(\sum_{k=0}^n \xi^{k+1}, \A^{-1} R^{2, n+1}\right)\notag\\
\geq& \rho^N+\frac{\Delta t}{2 D^*} \sum_{n=0}^{N}\3{\xi^{n+1}}^2-\left|\sum_{n=0}^{N}\left(\Delta t \sum_{k=0}^n \xi^{k+1},\A^{-1} R^{2, n+1}\right)\right|\notag\\
\geq& \rho^N+\frac{\Delta t}{2 D^*} \sum_{n=0}^{N}\3{\xi^{n+1}}^2-\Delta t \sum_{n=0}^{N} D^* \rho^n-\frac{1}{2R_a^2} \sum_{n=0}^{N}\3{R^{2, n+1}}^2.
\end{align}
Using equation~\eqref{eq-con22} and the inequalities
\( |\alpha \beta| \leq \varepsilon \alpha^2 + \frac{1}{4\varepsilon} \beta^2 \) and  
\( \left( \sum_{i=1}^n a_i \right)^2 \leq n \sum_{i=1}^n a_i^2 \)
together with Lemmas \ref{lem-3.2}-\ref{lem-3.4}, we obtain the following estimate
\begin{align}\label{eq-con-g4}
 |M_4|=  \left |\sum_{n=0}^{N} \langle\sum_{k=0}^n \Delta t \mathcal{H} (u \eta)^k,{\A^*}^{-1} \eta^{n+1}\rangle\right|\leq\frac{L_u^2R_dT\Delta t}{4\varepsilon h^2} \sum_{n=0}^{N}\beta^n+\frac{\varepsilon}{R_a^2}\sum_{n=0}^{N}\n{\eta^{n+1}}^2,
\end{align}
\begin{align}\label{eq-con-g5}
 M_5=  \left |\sum_{n=0}^{N}\langle\sum_{k=0}^n \Delta t R^{1, k+1}, \A^{-1} \eta^{n+1}\rangle \right|\leq \frac{T^2}{4\varepsilon}\sum_{n=0}^{N}\n{R^{1, n+1}}^2 +\frac{\varepsilon}{R_a^2}\sum_{n=0}^{N}\n{\eta^{n+1}}^2.
\end{align}
By substituting the estimates \eqref{eq-con-left1}-\eqref{eq-con-g2} and \eqref{eq-con-g3}-\eqref{eq-con-g5} into  \eqref{eq-com-all2}, we get
\begin{align}\label{eq-con-last}
 (1-\frac{2 \varepsilon}{R_a^2}) \beta^N+\rho^N \leq \frac{L_u^2R_dT}{4\varepsilon h^2} \sum_{n=0}^{N}\Delta t\beta^n +D^*\sum_{n=0}^{N}\Delta t   \rho^n
+\frac{T^2}{4\varepsilon}\sum_{n=0}^{N}\n{R^{1, n+1}}^2+\frac{1}{2R_a^2} \sum_{n=0}^{N}\3{R^{2, n+1}}^2.
\end{align}
Ensuring that $\varepsilon$ fulfills the inequality $1-\frac{2\varepsilon}{R_a^2}>0$, and applying Gronwall’s inequality to \eqref{eq-con-last}, we obtain
\begin{equation}\label{eq-last}
\beta^N\leq K \sum_{n=0}^{N}(\n{R^{1, n+1}}^2+\3{R^{2, n+1}}^2) .
\end{equation}
By multiplying \eqref{eq-last} by $\Delta t$, we obtain
\begin{equation}
\sum_{n=0}^{N} \Delta t\left\|\eta^{n+1}\right\|^2 \leq K \sum_{n=0}^{N} \Delta t\left(\left\|R^{1, n+1}\right\|^2+\3{R^{2, n+1}}^2\right) \leq K\left(h^{2s}+\Delta t^2\right).
\end{equation}
It completes the proof.
\end{proof}
\section{Numerical experiments}
This section evaluates the numerical formats for the contaminant transport equation with adsorption, starting with periodic and extending to Dirichlet boundary conditions. For cases with the nonlinear term $ \phi(c)$, Newton's method is applied at each time step.
\par
The following norms are employed to measure the errors of numerical solutions
\begin{align*}
    &\varepsilon _{c,2}=\n{C-c}, \quad \varepsilon _{c,\infty}=\underset{i}{\mathrm{max}}\left | C_i-c_i \right |,\\
    &\varepsilon _{z ,2}=\3{Z-z}, \quad \varepsilon_{z,\infty}=\underset{i}{\mathrm{max}}\left | Z_{i+\f}-z_{i+\f} \right |.
\end{align*}
In the case of periodic boundary conditions, the block-centered numerical schemes demonstrate mass conservation, with mass errors quantified by
\begin{align*}
\varepsilon _{Mass}= \left| \sum_{i=1}^{J} C_{i}^{n+1}h+\sum_{i=1}^{J} \Phi_{i}^{n+1}h-\sum_{i=1}^{J}C_{i}^{0}h-\sum_{i=1}^{J}\Phi_{i}^{0}h-\sum_{k=0}^{n} \sum_{i=1}^{J}
\Delta t(f^{k+1}_{i}h )\right |. 
\end{align*}
\subsection{Examples with periodic boundary condition}
To examine the effects of different adsorption models, we first consider the contaminant transport problem with adsorption described by the Langmuir isotherm.
\begin{example}
Consider a contaminant transport equation with adsorption on $[0, 2\pi]$ with periodic boundary conditions. The velocity field is $u(x) = \sin(2x)$, and the diffusion coefficient is $D(x) = 0.1(\cos(2x) + 2)$. The nonlinear adsorption term is $\phi(c) = \frac{5c}{1 + 6c}$. The exact solution is $c(x, t) = \exp(-t) \frac{\sin(2x) + 1}{2}$, and the source term $f(x, t)$ is derived accordingly.
\end{example}
In the numerical experiments, the proposed HOS1–HOS4 schemes are used for spatial discretization, with time integration mainly based on the forward Euler method to align with the theoretical analysis. To fully leverage the high spatial accuracy of HOS3 and HOS4 and substantially enhance computational efficiency, the Crank–Nicolson method is additionally employed for time stepping in these schemes:
\begin{align*}
\partial_t^{CN} v_i^{n+\f}=\frac{v_i^{n+1}-v_i^{n}}{\Delta t},\quad v_i^{n+\f}=\frac{v_i^{n+1}+v_i^{n}}{2}. 
\end{align*} 
Using the Crank-Nicolson (CN), we derive the following block-centered scheme:
\begin{align*}
&\A(m,a_2) \partial_t^{CN}  C_{i}^{n+\f}+\A(m,a_2) \partial_{t}^{CN}\Phi_{i}^{n+\f}+\delta(m) Z_{i}^{n+\f}   \\
&= -\A(m,a_2)\A(m^*,a_2^*)^{-1}\mathcal{H} (m^*)(u C)_{i}^{n+\f}+\A(m,a_2) f_{i}^{n+\f},\quad i= 1,2, \cdots, J,  \\
&\delta(m) C_{i+1 / 2}^{n+1}+\A(m,a_2)\left(\frac{Z}{D}\right)_{i+1 / 2}^{n+1}  = 0,  \quad i= 0,1, \cdots, J-1.
\end{align*}
\par
The spatial accuracy test results of the numerical schemes HOS1--HOS4 are presented in Tables~\ref{tab:1}--\ref{tab:2}. The results confirm that the schemes converge at their respective theoretical orders, with HOS1 and HOS2 achieving fourth-order accuracy, HOS3 achieving sixth-order accuracy, and HOS4 achieving eighth-order accuracy as the grid is refined.
\begin{table} [!h]
\caption{Convergence and Error Rates of $c$ and $z$ in 4th-Order Schemes  with $\Delta t = h^4$ (Euler  Method).}
\label{tab:1}       % Give a unique label
%\begin{tabular}{llllllllll}    %原来
 %\begin{tabularx}{\textwidth}{lXXXXXXXXX}  %这是我自己加的
 % \renewcommand{\arraystretch}{1.45}  % 设置行间距为1.5倍
 % \setlength{\tabcolsep}{3pt}  % 设置全局列间距
 \centering
\resizebox{0.8\textwidth}{!}{ % 将表格缩放到页面宽度
\renewcommand{\arraystretch}{1.3}  % 设置行间距为1.5倍
\begin{tabular}{cccccccccc}
\hline\noalign{\smallskip}
 & $J$ & $ \varepsilon _{c,\infty}$ & Rate & $\varepsilon _{c,2}$ & Rate & $\varepsilon _{z,\infty} $ & Rate & $\varepsilon _{z,2}$ & Rate \\
\noalign{\smallskip}\hline\noalign{\smallskip}
\multirow{4}{*}{HOS1} & 15 & 0.0204     & -      & 0.0290     & -      & 0.0110     & -      & 0.0200     & -      \\
                     & 20 & 7.3588e-03 & 3.5430 & 1.0082e-02 & 3.6695 & 3.5304e-03 & 3.9435 & 6.8604e-03 & 3.7157 \\
                     & 30 & 1.5883e-03 & 3.7815 & 2.1583e-03 & 3.8017 & 7.8868e-04 & 3.6965 & 1.4642e-03 & 3.8092 \\
                     & 40 & 5.1444e-04 & 3.9186 & 7.0239e-04 & 3.9021 & 2.5135e-04 & 3.9750 & 4.7744e-04 & 3.8953 \\ \noalign{\smallskip}\hline\noalign{\smallskip}
\multirow{4}{*}{HOS2} & 15 & 0.0120     & -      & 0.0144     & -      & 0.0038     & -      & 5.1740e-03 & -      \\
                     & 20 & 4.0929e-03 & 3.7501 & 4.5690e-03 & 3.9821 & 1.0140e-03 & 4.6142 & 1.5100e-03 & 4.2809 \\
                     & 30 & 8.2414e-04 & 3.9527 & 9.1276e-04 & 3.9721 & 1.8356e-04 & 4.2152 & 2.9500e-04 & 4.0272 \\
                     & 40 & 2.7216e-04 & 3.8513 & 2.9018e-04 & 3.9835 & 6.1066e-05 & 3.8256 & 9.3521e-05 & 3.9932 \\
\noalign{\smallskip}\hline
\end{tabular}
}
\end{table}

\begin{table} [!h]
\caption{Error and Convergence of $c$ and $z$ in 6th/8th-Order Schemes with $\Delta t = h^3/h^4$ (Crank Method).}
\label{tab:2}       % Give a unique label
\centering
\resizebox{0.8\textwidth}{!}{ % 将表格缩放到页面宽度
\renewcommand{\arraystretch}{1.3}  % 设置行间距为1.5倍
\begin{tabular}{cccccccccc}
\hline\noalign{\smallskip}
 & $J$ & $ \varepsilon _{c,\infty}$ & Rate & $\varepsilon _{c,2}$ & Rate & $\varepsilon _{z,\infty}$ & Rate & $\varepsilon _{z,2} $ & Rate \\ \noalign{\smallskip}\hline\noalign{\smallskip}
\multirow{4}{*}{HOS3} & 15 & 1.2333e-03 & - & 1.8831e-03 & - & 8.1225e-04 & - & 1.3707e-03 & - \\
 & 20 & 1.9516e-04 & 6.4085 & 2.6868e-04 & 6.7685 & 1.1721e-04 & 6.7291 & 1.8785e-04 & 6.9085 \\
 & 25 & 4.6628e-05 & 6.4157 & 6.3831e-05 & 6.4411 & 2.8498e-05 & 6.3374 & 4.4149e-05 & 6.4895 \\
 & 30 & 1.4609e-05 & 6.3656 & 2.0043e-05 & 6.3534 & 9.8735e-06 & 5.8137 & 1.3850e-05 & 6.3586 \\\noalign{\smallskip}\hline\noalign{\smallskip}
\multirow{4}{*}{HOS4} & 15 & 3.3039e-04 & - & 5.0004e-04 & - & 2.2142e-04 & - & 3.5947e-04 & - \\
 & 20 & 2.8110e-05 & 8.5655 & 3.9951e-05 & 8.7842 & 1.8008e-05 & 8.7224 & 2.7770e-05 & 8.9011 \\
 & 25 & 4.4378e-06 & 8.2726 & 5.9895e-06 & 8.5041 & 2.7077e-06 & 8.4908 & 4.1269e-06 & 8.5435 \\
 & 30 & 9.0865e-07 & 8.6987 & 1.3075e-06 & 8.3475 & 6.4654e-07 & 7.8555 & 9.0106e-07 & 8.3463 \\ 
\noalign{\smallskip}\hline
\end{tabular}
}
\end{table}
Furthermore, Table \ref{tab:3} presents the mass errors of  $c$ at various times when $J=30$. It is evident that the mass errors in the scenarios listed reach 
$10^{-15}$, indicating the property of mass conservation and aligning with the theoretical results.
\begin{table}[!h]
\caption{Mass Error $\varepsilon _{Mass}$ of Numerical Solution $C$(Euler  Method).}
\label{tab:3} 
% \scriptsize  % 
\centering
\renewcommand{\arraystretch}{1.3}  % 设置行间距为1.5倍
\scalebox{0.8}{
\begin{tabularx}{1\textwidth}{XXXXXX}
\hline\noalign{\smallskip}
 &  & $t = 0.2$ & $t = 0.4$ & $t = 0.6$ & $t = 0.8$ \\ \noalign{\smallskip}\hline\noalign{\smallskip}
\multirow{2}{*}{$\Delta t=1/200$} & HOS1 & 6.6613e-15 & 8.8818e-15 & 1.6431e-14 & 2.2204e-14 \\
                                  & HOS2 & 2.2204e-15 & 5.3291e-15 & 2.2204e-15 & 1.3323e-15 \\ \noalign{\smallskip}\hline\noalign{\smallskip}
$\Delta t=1/350$                  & HOS3 & 7.7716e-16 & 6.6613e-16 & 3.7748e-15 & 9.3259e-15 \\ \noalign{\smallskip}\hline\noalign{\smallskip}
$\Delta t=1/500$    & HSO4 & 1.9984e-15 & 7.1054e-15 & 8.4377e-15 & 1.3323e-14\\ \noalign{\smallskip}\hline
\end{tabularx}
}
\end{table}
To investigate the influence of parameter variation on numerical accuracy, we adopt a parameterization strategy guided by the stage truncation error expressions~\eqref{eq-opp} and~\eqref{eq-op}. Specifically, we set
 \(a_2^* = \frac{{m^*}^2}{72} - \frac{1}{30}\) and \(a_2= \frac{{m}^2}{288} - \frac{3}{640}\) so that the resulting schemes achieve sixth-order accuracy in space. This setting ensures structural consistency of the discretization and effectively reduces the degrees of freedom to two essential parameters, \(m^*\) and \(m\), enabling a focused analysis of their impact on the accuracy of the solution and flux.
Time discretization uses the Crank method with \(J = 20\) and \(\Delta t = h^3\). We treat \(m\) and \(m^*\) as variables and examine how the \(L^2\)-norm errors of solution and flux vary. As shown in Figure~\ref{fig_4}, both errors exhibit similar trends with respect to \(m\) and \(m^*\), reaching minima at \(m ^*= 9\sqrt{\frac{3}{119}}\), \(m = 2\sqrt{\frac{6}{7}}\), 
matching the theoretical HOS4 scheme, which corresponds to the highest accuracy theoretically achievable within the parameterized discretization framework. This observation indicates the consistency between theory and numerical results.
\begin{figure}[!h]
\centering % 使图像居中
\includegraphics[width=0.8\textwidth]{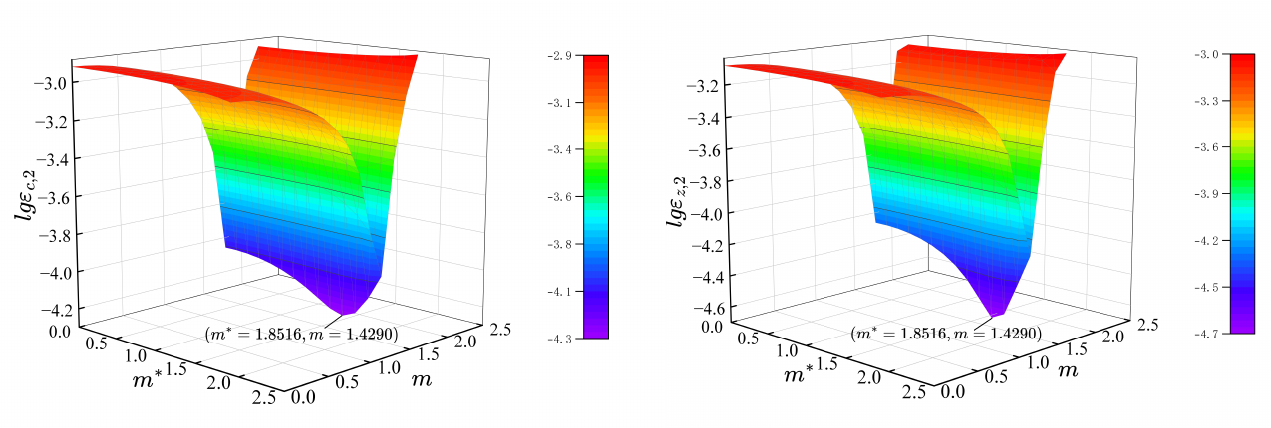}
\caption{Trend of the $L^2$ norm error with parameter.}
\label{fig_4}       % Give a unique label
\end{figure}
\par
Next, to further explore the influence of adsorption behavior, we consider the contaminant transport problem governed by the Freundlich isotherm.
\begin{example}
We consider a nonlinear contaminant transport equation with adsorption on \( x \in [0,\pi] \), with the exact solution \( c = 3^{\cos(2x + t) - 1} \). The velocity and the diffusion coefficients are \( u(x) = \cos(2x) \) and \( D(x) = \frac{1}{2} \sin(2x) + 1 \), respectively. The nonlinear adsorption is \( \phi(c) = c^{p} \) with \( p = \frac{1}{3} \). The source term \( f(x,t) \) follows from the exact solution. Results are evaluated at \( T = 1 \).
\end{example}
Numerical results confirm the convergence of schemes HOS1–HOS4. As shown in Tables~\ref{tab:4}–\ref{tab:5}, they achieve fourth-, sixth-, and eighth-order accuracy with grid refinement. Table~\ref{tab:6} reports mass errors of \( c \) for \( J = 20 \), on the order of \(10^{-15}\), verifying mass conservation.

Using the same schemes and parameters as before (\(a_2^* = \frac{{m^*}^2}{72} - \frac{1}{30}\), \(a_2= \frac{{m}^2}{288} - \frac{3}{640}\) ), time is discretized by the Crank method with \(J = 15\), \(\Delta t = h^3\). As shown in Figure~\ref{fig_5}, the \(L^2\)-norm errors of the solution and flux reach minima at \(m^* = 9\sqrt{\frac{3}{119}}\) and \(m = 2\sqrt{\frac{6}{7}}\), matching HOS4 theory.

\begin{table} [!h]
\caption{Convergence and Error Rates of $c$ and $z$ in 4th-Order Schemes with $\Delta t = h^4$ (Euler  Method).}
\label{tab:4}       % Give a unique label
  % 设置行间距为1.5倍
  % 设置全局列间距
\centering
\resizebox{0.8\textwidth}{!}{ % 将表格缩放到页面宽度
\renewcommand{\arraystretch}{1.3}  % 设置行间距为1.5倍
\begin{tabular}{cccccccccc}
\hline\noalign{\smallskip}
 & J & $ \varepsilon _{c,\infty}$ & Rate & $\varepsilon _{c,2}$ & Rate & $\varepsilon _{z,\infty} $ & Rate & $\varepsilon _{z,2}$ & Rate \\ \noalign{\smallskip}\hline\noalign{\smallskip}
\multirow{4}{*}{HOS1} & 15 & 2.3613e-03 & -      & 1.7718e-03 & -      & 6.2628e-03 & -                          & 4.7487e-03 & -      \\
                     & 20 & 8.0406e-04 & 3.7447 & 5.7925e-04 & 3.8862 & 1.9677e-03 & 4.0244                     & 1.5723e-03 & 3.8421 \\
                     & 30 & 1.7577e-04 & 3.7500 & 1.1743e-04 & 3.9360 & 4.2238e-04 & 3.7950                     & 3.2161e-04 & 3.9140 \\
                     & 40 & 5.6512e-05 & 3.9443 & 3.7517e-05 & 3.9664 & 1.3759e-04 & 3.8988                     & 1.0307e-04 & 3.9556 \\ \noalign{\smallskip}\hline\noalign{\smallskip}
\multirow{4}{*}{HOS2} & 15 & 4.9165e-04 & -      & 4.9563e-04 & -      & 2.5989e-03 & -                          & 2.1828e-03 & -      \\
                     & 20 & 1.5526e-04 & 4.0067 & 1.5621e-04 & 4.0136 & 9.7602e-04 & \multicolumn{1}{l}{3.4043} & 6.6856e-04 & 4.1130 \\
                     & 30 & 3.3146e-05 & 3.8084 & 3.0751e-05 & 4.0084 & 1.7970e-04 & \multicolumn{1}{l}{4.1734} & 1.2863e-04 & 4.0648 \\
                     & 40 & 1.0371e-05 & 4.0389 & 9.7167e-06 & 4.0046 & 5.8331e-05 & \multicolumn{1}{l}{3.9112} & 4.0309e-05 & 4.0336 \\ 
                     \noalign{\smallskip}\hline
\end{tabular}
}
\end{table}

\begin{table} [!h]
\caption{Error and Convergence of $c$ and $z$ in 6th/8th-Order Schemes with $\Delta t = h^3/h^4$ (Crank Method).}
\label{tab:5}       % Give a unique label
\centering
\resizebox{0.8\textwidth}{!}{ % 将表格缩放到页面宽度
\renewcommand{\arraystretch}{1.3}  % 设置行间距为1.5倍
\begin{tabular}{cccccccccc}
\hline\noalign{\smallskip}
 & $J$ & $ \varepsilon _{c,\infty}$ & Rate & $\varepsilon _{c,2}$ & Rate & $\varepsilon _{z,\infty}$ & Rate & $\varepsilon _{z,2} $ & Rate \\ \noalign{\smallskip}\hline\noalign{\smallskip}
\multirow{4}{*}{HOS3} & 15 & 9.9351e-05 & - & 5.5122e-05 & - & 3.9859e-04 & - & 3.1609e-04 & - \\
 & 20 & 1.5950e-05 & 6.3584 & 8.7796e-06 & 6.3859 & 8.0259e-05 & 5.5710 & 4.9766e-05 & 6.4261 \\
 & 25 & 3.9811e-06 & 6.2198 & 2.1964e-06 & 6.2096 & 1.8953e-05 & 6.4681 & 1.2319e-05 & 6.2568 \\
 & 30 & 1.2969e-06 & 6.1518 & 7.1657e-07 & 6.1435 & 6.1578e-06 & 6.1662 & 3.9956e-06 & 6.1758 \\ \noalign{\smallskip}\hline\noalign{\smallskip}
\multirow{4}{*}{HOS4} & 15 & 1.5183e-05 & - & 1.0574e-05 & - & 4.4186e-05 & - & 3.0921e-05 & - \\
 & 20 & 1.3204e-06 & 8.4893 & 8.2063e-07 & 8.8849 & 3.7492e-06 & 8.5749 & 2.3417e-06 & 8.9702 \\
 & 25 & 1.7973e-07 & 8.9369 & 1.2478e-07 & 8.4409 & 5.9376e-07 & 8.2585 & 3.6709e-07 & 8.3042 \\
 & 30 & 4.4044e-08 & 7.7133 & 2.7629e-08 & 8.2694 & 1.2765e-07 & 8.4313 & 8.2737e-08 & 8.1720 \\ 
 \noalign{\smallskip}\hline
\end{tabular}}
\end{table}
\begin{table}[!h]
\caption{Mass Error $\varepsilon _{Mass}$ of Numerical Solution $C$ (Euler  Method).}
\label{tab:6} 
\centering
\renewcommand{\arraystretch}{1.3}  % 设置行间距为1.5倍
\scalebox{0.8}{
\begin{tabularx}{1\textwidth}{XXXXXX}
\hline\noalign{\smallskip}
                                  &     & $ t=0.2$   & $t=0.4$    & $t=0.6$    & $t=0.8$    \\ \noalign{\smallskip}\hline\noalign{\smallskip}
$\Delta t=1/500$      & HOS1 & 4.4855e-16 &3.1050e-15 & 4.8800e-15
& 5.7827e-15\\ \noalign{\smallskip}\hline\noalign{\smallskip}
$\Delta t=1/750$    & HOS2 & 3.1096e-15 & 1.3384e-15 & 8.8990e-16 & 4.5497e-16\\ \noalign{\smallskip}\hline\noalign{\smallskip}
$\Delta t=1/500$     & HOS3 & 1.3367e-15 & 3.1050e-15 & 4.4359e-15 & 2.2299e-15 \\ \noalign{\smallskip}\hline\noalign{\smallskip}
$\Delta t=1/1000$     & HOS4 & 5.3254e-15 & 7.5502e-15 & 7.1036e-15 & 1.1981e-14\\ \noalign{\smallskip}\hline
\end{tabularx}
}
\end{table}

\begin{figure}[!h]
\centering % 使图像居中
  \includegraphics[width=0.8\textwidth]{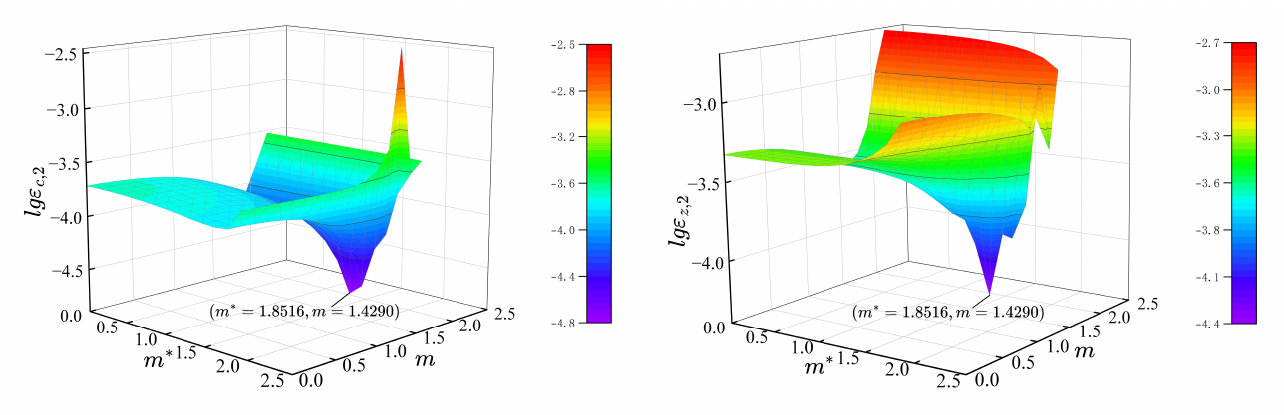}
\caption{Trend of the $L^2$ norm error with parameter.}
\label{fig_5}       % Give a unique label
\end{figure}

%%%%%%%%%%%%%%%%%%%%%%%狄利克雷算例
\subsection{Examples with Dirichlet boundary condition}
To extend the analysis to Dirichlet boundary conditions, we begin with a contaminant transport problem featuring linear adsorption.
\begin{example}
We consider a Dirichlet problem with parameters from: the constant velocity \(u = 0.15\), the diffusion \(D = 0.135\), and a linear adsorption \(\phi(c) = K_dc\), \(K_d = 0.7\). In domain \([0,4]\), the exact solution \(c(x,t) = e^{t}\cos^{2}(x)\) determines \(f(x,t)\) and the boundary data. Results are at \(T = 1\).
\end{example}
In the numerical experiments, the time step is set as \(\Delta t = h^4\). Table~\ref{tab:linear} presents the spatial accuracy results for HOS1-D and HOS2-D, both achieving fourth-order convergence for solution and flux under mesh refinement, fully consistent with theoretical analysis.
\begin{table} [!h]
\caption{Error and Convergence of $c$ and $z$ with $\Delta t = h^4$ (Euler  Method)}
\label{tab:linear}       % Give a unique label
\centering
\resizebox{0.8\textwidth}{!}{ % 将表格缩放到页面宽度
\renewcommand{\arraystretch}{1.3}  % 设置行间距为1.5倍
\begin{tabular}{cccccccccc}
\hline\noalign{\smallskip}
 & $J$ & $ \varepsilon _{c,\infty}$ & Rate & $\varepsilon _{c,2}$ & Rate & $\varepsilon _{z,\infty}$ & Rate & $\varepsilon _{z,2} $ & Rate \\ \noalign{\smallskip}\hline\noalign{\smallskip}
\multirow{4}{*}{HOS1-D} & 10 & 2.2590e-02 & - & 2.7665e-02 & - & 6.8122e-03 & - & 5.7897e-03 & - \\
 & 15 & 4.4126e-03 & 4.0276 & 5.1533e-03 & 4.1447 & 1.6302e-03 & 3.5269 & 1.1868e-03 & 3.9087 \\
 & 20 & 1.3908e-03 & 4.0133 & 1.6046e-03 & 4.0558 & 5.3730e-04 & 3.8580 & 3.7571e-04 & 3.9981 \\
 & 30 & 2.7541e-04 & 3.9939 & 3.1384e-04 & 4.0243 & 1.0445e-04 & 4.0395 & 7.2746e-05 & 4.0492 \\ \noalign{\smallskip}\hline\noalign{\smallskip}
\multirow{4}{*}{HOS2-D} & 10 & 2.1071e-02 & - & 2.7986e-02 & - & 8.4376e-03 & - & 6.5138e-03 & - \\
 & 15 & 4.1377e-03 & 4.0145 & 5.1060e-03 & 4.1959 & 1.7969e-03 & 3.8145 & 1.3046e-03 & 3.9658 \\
 & 20 & 1.3075e-03 & 4.0045 & 1.5750e-03 & 4.0884 & 5.6382e-04 & 4.0290 & 4.1021e-04 & 4.0219 \\
 & 30 & 2.5779e-04 & 4.0047 & 3.0543e-04 & 4.0454 & 1.0964e-04 & 4.0386 & 7.8850e-05 & 4.0672  \\ \noalign{\smallskip}\hline
\end{tabular}}
\end{table}
\par
Next, we consider the Freundlich isotherm case, where \(\phi'(c)\) tends to infinity as \(c \to 0\), resulting in a degenerate adsorption problem.
\begin{example}
Consider a Dirichlet problem on \([-3,3]\), with the analytical solution \(c(x,t) = exp(-t)\tanh^2(2x)\). Set \(u = x\), \(D = x^2 + 1\), and \(\phi(c) = c^p\) with \(p = \frac{1}{3}\). The source term \(f(x,t)\) and boundary conditions follow directly from the analytical solution. Results are evaluated at the terminal time \(T = 1\).
\end{example}
\par
Since \( {\phi}' (0)=\infty \), the problem is degenerate. Thus, we approximate \( \phi(c) \) by a regularized function \( \phi_\varepsilon(c) \) with bounded slope \cite{barrett1997finite,radu2010newton}, defined as
\begin{equation}
\phi_{\varepsilon}(c) = 
\begin{cases}
\phi(c), & c > \varepsilon, \\
p \varepsilon^{p-1}c + (1-p)\varepsilon^p, & c \in [0, \varepsilon],
\end{cases}
\end{equation}
The regularization parameter is set to \(\varepsilon = 10^{-10}\).
\par
In the numerical experiments with \(\Delta t = h^4\), Table~\ref{tab:END2} shows that both HOS1-D and HOS2-D achieve fourth-order convergence for solution and flux under mesh refinement, consistent with theoretical analysis. Both schemes demonstrate robust performance when applied to this degenerate problem.
\begin{table} [!h]
\caption{Error and Convergence Rates of $ c$ and $z $ with $\Delta t=h^4$ (Euler  Method)}
\label{tab:END2}       % Give a unique label
\centering
\resizebox{0.8\textwidth}{!}{ % 将表格缩放到页面宽度
\renewcommand{\arraystretch}{1.3}  % 设置行间距为1.5倍
\begin{tabular}{cccccccccc}
\hline
 & J & $ \varepsilon _{c,\infty}$ & Rate & $\varepsilon _{c,2}$ & Rate & $\varepsilon _{z,\infty} $ & Rate & $\varepsilon _{z,2}$ & Rate \\ \noalign{\smallskip}\hline\noalign{\smallskip}
\multirow{4}{*}{HOS1-D} & 30 & 1.4512e-03 & -      & 7.5479e-04 & -      & 3.2275e-03 & -      & 2.7762e-03 & -      \\
                        & 40 & 4.4773e-04 & 4.0876 & 2.3130e-04 & 4.1113 & 1.2092e-03 & 3.4127 & 8.8206e-04 & 3.9856 \\
                        & 50 & 1.7951e-04 & 4.0958 & 9.2360e-05 & 4.1140 & 5.4598e-04 & 3.5631 & 3.6264e-04 & 3.9834 \\
                        & 60 & 8.4775e-05 & 4.1150 & 4.3593e-05 & 4.1180 & 2.6082e-04 & 4.0519 & 1.7500e-04 & 3.9962 \\ \noalign{\smallskip}\hline\noalign{\smallskip}
\multirow{4}{*}{HOS2-D} & 30 & 6.5120e-04 & -      & 7.3638e-04 & -      & 1.3385e-03 & -      & 1.9580e-03 & -      \\
                        & 40 & 2.1917e-04 & 3.7853 & 2.5735e-04 & 3.6544 & 4.5766e-04 & 3.7303 & 6.7153e-04 & 3.7197 \\
                        & 50 & 9.4190e-05 & 3.7847 & 1.0834e-04 & 3.8772 & 1.9206e-04 & 3.8913 & 2.8137e-04 & 3.8983 \\
                        & 60 & 4.6044e-05 & 3.9256 & 5.2759e-05 & 3.9466 & 9.3573e-05 & 3.9440 & 1.3678e-04 & 3.9561 \\\noalign{\smallskip}\hline
\end{tabular}%
}
\end{table}
\par
Next, we consider another adsorption model, the Langmuir isotherm, to study the contaminant transport problem.
\begin{example}
We consider a numerical problem with Dirichlet boundary conditions on $[0,6]$, with velocity $u = x$, diffusion $D = x/10$, and adsorption $\phi(c) = \frac{c}{1+c}$. The exact solution $c = \exp(-t) \sin^2(x)$ determines $f(x,t)$ and the boundary conditions. All results correspond to the final time $T = 1$.
\end{example}
Numerical results in Table \ref{tab:END3} demonstrate that with $\Delta t = h^4$, both HOS1-D and HOS2-D schemes achieve fourth-order convergence for solution and flux under grid refinement.
\begin{table} [!h]
\caption{Error and Convergence Rates of $ c$ and $z $ with $\Delta t=h^4$ (Euler  Method)}
\label{tab:END3}       % Give a unique label
\centering
\resizebox{0.8\textwidth}{!}{ % 将表格缩放到页面宽度
\renewcommand{\arraystretch}{1.3}  % 设置行间距为1.5倍
\begin{tabular}{cccccccccc}
\hline\noalign{\smallskip}
 & $J$ & $ \varepsilon _{c,\infty}$ & Rate & $\varepsilon _{c,2}$ & Rate & $\varepsilon _{z,\infty}$ & Rate & $\varepsilon _{z,2} $ & Rate \\ \noalign{\smallskip}\hline\noalign{\smallskip}
\multirow{4}{*}{HOS1-D} & 10 & 7.5386e-02 & - & 8.4771e-02 & - & 3.1093e-02 & - & 3.5201e-02 & - \\
 & 20 & 2.0080e-03 & 5.2305 & 2.9878e-03 & 4.8264 & 2.1986e-03 & 3.8219 & 1.8707e-03 & 4.2340 \\
 & 25 & 8.2933e-04 & 3.9628 & 1.2654e-03 & 3.8503 & 8.5935e-04 & 4.2100 & 7.4921e-04 & 4.1006 \\
 & 30 & 4.1171e-04 & 3.8410 & 6.2381e-04 & 3.8792 & 3.8275e-04 & 4.4361 & 3.5213e-04 & 4.1412 \\ \noalign{\smallskip}\hline\noalign{\smallskip}
\multirow{4}{*}{HOS2-D} & 10 & 5.8538e-02 & - & 7.3733e-02 & - & 4.0306e-02 & - & 5.0032e-02 & - \\
 & 20 & 2.8440e-03 & 4.3634 & 3.9571e-03 & 4.2198 & 1.5021e-03 & 4.7460 & 1.9309e-03 & 4.6955 \\
 & 25 & 1.0050e-03 & 4.6618 & 1.5646e-03 & 4.1582 & 6.3435e-04 & 3.8630 & 7.9308e-04 & 3.9878 \\
 & 30 & 4.6267e-04 & 4.2546 & 7.3910e-04 & 4.1133 & 3.0795e-04 & 3.9636 & 3.8468e-04 & 3.9684 \\ \noalign{\smallskip}\hline
\end{tabular}}
\end{table}
\par
As the final example, we consider a realistic application that demonstrates the practical effectiveness of the proposed numerical schemes.
\begin{example}
Pb migration is simulated in 5, 10, and 15-m soil columns over 1800 days. The initial Pb concentration is set to zero, assuming Langmuir adsorption:
 \( \phi(c) = K_{L} \frac{S_{m} c}{1 + K_{L} c} \).
Boundary conditions specify a constant inlet concentration of 100~mg/L and a zero-flux condition at the outlet. Model parameters are taken from~\cite{chetti2024simulating} (see Table~\ref{tab:ee}). 
\begin{table}[h!]
\caption{Simulation Parameters}
 \renewcommand{\arraystretch}{1.5}  % 设置行间距为1.5倍
\label{tab:ee}
\centering
\scalebox{0.8}{  
\begin{tabular}{cccccc}
\hline\noalign{\smallskip}
$\rho\,(kg/m^3)$ & $u\,(m/day)$ & $D\,(m^2/day)$ & $n$ & $K_L\,(l/g)$ & $S_m\,(mg/g)$ \\ \noalign{\smallskip}\hline\noalign{\smallskip}
1500             & 0.012        & 0.17477        & 0.3 & 2.6          & $3\times 10^{-4}$ \\ \noalign{\smallskip}\hline
\end{tabular}
}
\end{table}
\end{example}
\begin{figure}[!h]
\centering % 使图像居中
\includegraphics[width=0.75\textwidth]{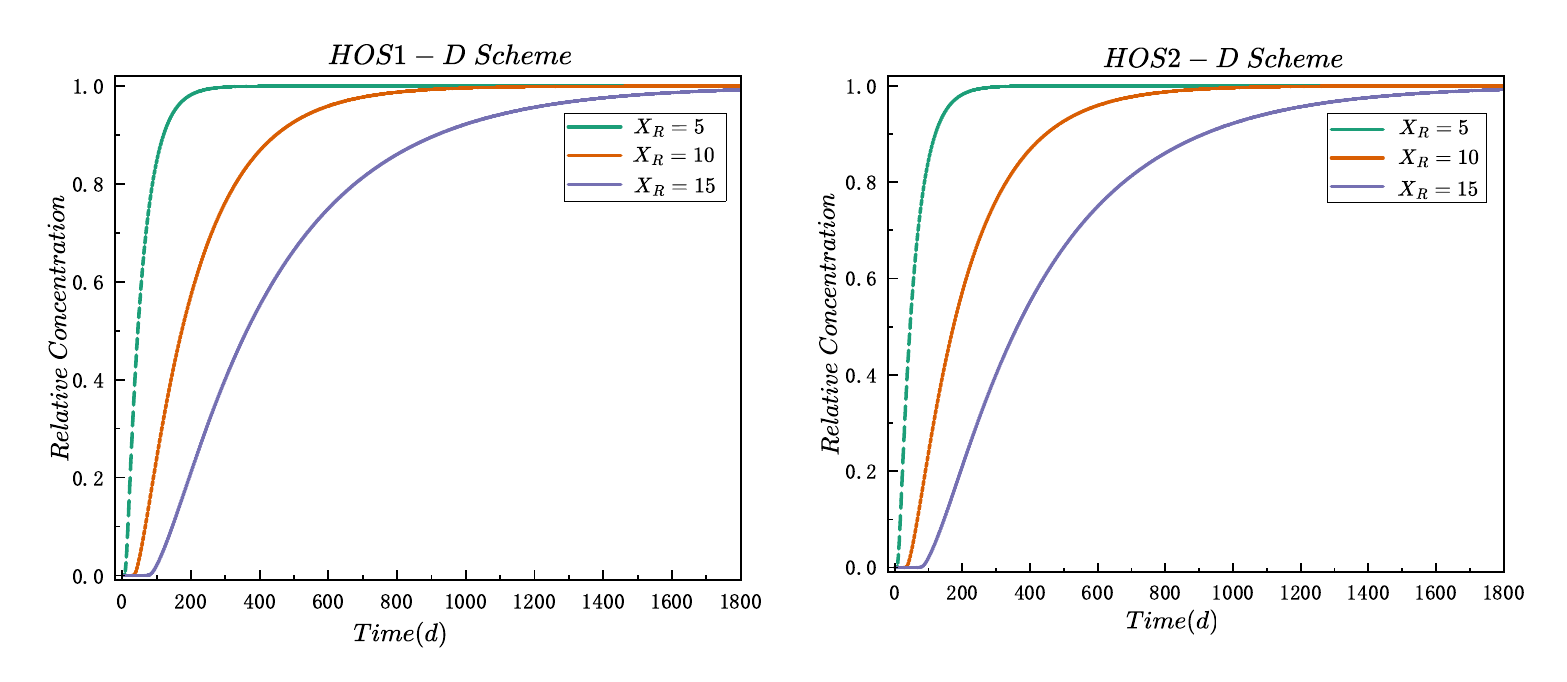}
\caption{Temporal evolution of relative lead concentration \((C/C_0)\) in soil columns of varying lengths (HOS1-D and HOS2-D).}
\label{fig_6}       % Give a unique label
\end{figure}
To enforce the zero-flux boundary condition in a Dirichlet-type scheme, the outlet concentration is set equal to that of the nearest interior node. This approach effectively approximates the zero-gradient condition while ensuring the consistency and stability of the scheme. The simulation is conducted with a spatial resolution of 0.5 m and a time step of 1 day.
According to the simulation results in Figure \ref{fig_6}, the relative concentration of lead ($C/C_0$) in soil columns of lengths $X_R = 5$, 10, and 15 m under both HOS1-D and HOS2-D formats exhibits the same time-dependent trend. With increasing column length, the time for Pb to reach the outlet increases, and the concentration rise slows down, indicating that a longer pathway leads to greater Pb retention and slower migration. Over time, the relative concentrations converge to the inlet value $C_0$, suggesting the system reaches a steady state under continuous feed, thereby confirming the appropriateness of the boundary conditions. The two formats yield highly consistent concentration profiles across all column lengths, demonstrating favorable numerical stability and accuracy. Overall, the results highlight the significant influence of column length on Pb transport and validate the reliability of the adopted high-order numerical schemes in modeling adsorption-driven migration.
\section{Conclusion} 
This paper develops a class of high-order conservative schemes for contaminant transport with equilibrium adsorption, based on the Integral Method with Variational Limit on block-centered grids.
The scheme incorporates four parameters, enabling fourth-, sixth-, and eighth-order formulations within a unified framework, addressing the lack of a unified construction framework for high-order schemes.
Theoretical analysis under periodic boundary conditions confirms the stability, convergence, and mass conservation of the schemes. 
Numerical experiments investigate the influence of parameter variations on computational errors and reveal the intrinsic relationship between parameters and schemes of different orders.
The results confirm that the schemes maintain mass conservation, with the convergence rates of both solutions and fluxes aligning with theoretical predictions.
To enhance the applicability of the method, two fourth-order compact boundary treatments are designed to ensure uniform accuracy near boundaries and in the interior. Numerical tests covering multiple adsorption models validate the effectiveness of the proposed approach. 
This work establishes a unified and structurally consistent framework compatible with various adsorption models, providing a solid foundation for the design and optimization of high-order numerical methods.
\section*{CRediT authorship contribution statement}
\textbf{He Liu}: Methodology, Writing -- original draft, Writing -- review \& editing, Visualization. 
\textbf{Xiongbo Zheng}: Investigation, Methodology, Project administration, Supervision, Writing -- review \& editing.
\textbf{Xiaole Li}: Investigation, Methodology, Project administration, Supervision, Writing -- review \& editing.
\textbf{Mingze Ji}: Visualization.
\section*{Declaration of competing interest}
All authors consent to this submission and declare no conflicts of interest.
\section*{Data availability}
Data will be made available on request.
\section*{Acknowledgments}
This research did not receive any specific grant from funding agencies in the public, commercial, or not-for-profit sectors.
\appendix
\section{Truncation Error }
Consider the operator \( \mathcal{A}(m, a_2)v_i := a_2 v_{i-2} + a_1 v_{i-1} + a_0 v_i + a_1 v_{i+1} + a_2 v_{i+2} \). Let \( v = w_x \), so that \( v_i = w_x(x_i) \). Then \( \mathcal{A}(m, a_2) \) approximates the composite operator \( \mathcal{H}(m)w_i := \frac{1}{h} \left( d_2 w_{i+2} + d_1 w_{i+1} - d_1 w_{i-1} - d_2 w_{i-2} \right) \).
Expanding \( v_{i\pm j} = w_x(x_i \pm jh) \) at \( x_i \) via Taylor series:
\begin{equation}\label{eq-tar}
v_{i\pm j} = \sum_{k=0}^{8} \frac{(\pm j h)^k}{k!} w^{(k+1)}(x_i).
\end{equation}
Substituting \eqref{eq-tar} into $\mathcal{A}(m, a_2)v_i$ yields:
\begin{align}
\mathcal{A}(m, a_2)v_i =& (2 a_2 + 2 a_1 + a_0) w_{i}^{(1)} +  (4 a_2 + a_1) w_{i}^{(3)} h^2 +  \left( \frac{4}{3} a_2 + \frac{1}{12} a_1 \right) w_{i}^{(5)} h^4 \notag \\
& +  \left( \frac{8}{45} a_2 + \frac{1}{360} a_1 \right) w_{i}^{(7)} h^6 + \left( \frac{4}{315} a_2 + \frac{1}{20160} a_1 \right)  w_{i}^{(9)} h^8+\mathcal{O}(h^9).\label{A-1}
\end{align}
Similarly, expanding the function inside \(\mathcal{H}(m) w_i\) at \(x_i\) by Taylor series, we have
\begin{align}
\mathcal{H}(m)w_i =  (2d_1 + 4d_2)w_{i}^{(1)}h  +  \left( \frac{1}{3} d_1  + \frac{8}{3} d_2 \right) w_{i}^{(3)}h^3 + \left( \frac{1}{60} d_1 + \frac{8}{15} d_2  \right) w_{i}^{(5)}h^5 
 +  \left( \frac{1}{2520} d_1  + \frac{16}{315} d_2 \right)   w_{i}^{(7)}h^7 +\mathcal{O}(h^9).\label{A-2}
\end{align}
Subtracting equation \eqref{A-1} from \eqref{A-2} and substituting the coefficients \(a\) and \(d\) as defined in the main text, we obtain
\begin{align*}
\mathcal{A}(m, a_2)v_i - \mathcal{H}(m)w_i 
=  \left(\frac{1}{30} + a_2 - \frac{{m}^2}{72}\right)w_{i}^{(5)}h^4 + \left(\frac{15+630a_2 - 7 {m}^2}{3780}\right) w_{i}^{(7)}h^6 
+ \left(\frac{3024a_2 + {m}^2}{241920}\right)w_{i}^{(9)}h^8 +\mathcal{O}(h^9),
\end{align*}
as shown in \eqref{eq-opp}.\\
Similarly, by performing a Taylor expansion of the function inside \(\delta(m) w_i\) at the point \(x_i\), we obtain
\begin{align}
   \delta(m)w_{i}=(2b_1 + 6b_2) w_{i}^{(1)}h  + (\frac{9}{8}b_2+\frac{1}{24})w_{i}^{(3)}h^3 +( \frac{81}{640}b_2+\frac{1}{1920}b_1 )w_{i}^{(5)}h^5 
 + h^7 (  \frac{243}{35840}b_2+\frac{1}{322560} b_1 )w_{i}^{(7)}h^7 +\mathcal{O}(h^9).\label{B-1}
\end{align}
Subtracting \eqref{A-1} from \eqref{B-1} and substituting the coefficients yields the expression appearing in the main text as \eqref{eq-op}.
\bibliographystyle{elsarticle-num}      % mathematics and physical sciences
\bibliography{mybib}   % name your BibTeX data base
%%%%%%%%%%%%%%%%%%%%%%%%%%%%%%%%%%%%%%%%%%%%%%%%%%%%%%%%%%%%%%%%%%%%%%
\end{document}